\documentclass[12pt]{amsart}
\usepackage{tikz}
\tikzset{
  treenode/.style = {shape=rectangle, rounded corners,
                     draw, align=center,
                     top color=white, bottom color=blue!20},
  root/.style     = {treenode, font=\Large, bottom color=red!30},
  env/.style      = {treenode, font=\ttfamily\normalsize},
  dummy/.style    = {circle,draw}
}

\usepackage{geometry, tikz-cd}                % See geometry.pdf to learn the layout options. There are lots.
\usepackage{graphicx}
\usepackage{amssymb}
\usepackage{epstopdf}
\usepackage{xcolor}
\usepackage{ulem}
\usepackage{cancel}
\usepackage[percent]{overpic}
\usepackage{stackengine}
\usepackage{verbatim}

\newcommand{\NN}{\mathbb{N}}

\newcommand{\ZZ}{\mathbb{Z}}

\newcommand{\Mod}{\operatorname{Mod}}

\newcommand{\Map}{\operatorname{MCG}}
\newcommand{\Teich}{\operatorname{Teich}}
\definecolor{lightgrey}{gray}{.85}

\numberwithin{equation}{section}

\theoremstyle{plain}
\newtheorem{thm}{Theorem}[section]
\newtheorem{cor}[thm]{Corollary}
\newtheorem{prop}[thm]{Proposition}
\newtheorem{lem}[thm]{Lemma}

\newtheorem{rem}[thm]{Remark}

\newtheorem{definition}[thm]{Definition}

\makeatletter
\@namedef{subjclassname@2020}{%
  \textup{2020} Mathematics Subject Classification}
\makeatother

\makeatletter
\def\moverlay{\mathpalette\mov@rlay}
\def\mov@rlay#1#2{\leavevmode\vtop{%
   \baselineskip\z@skip \lineskiplimit-\maxdimen
   \ialign{\hfil$\m@th#1##$\hfil\cr#2\crcr}}}
\newcommand{\charfusion}[3][\mathord]{
    #1{\ifx#1\mathop\vphantom{#2}\fi
        \mathpalette\mov@rlay{#2\cr#3}
      }
    \ifx#1\mathop\expandafter\displaylimits\fi}
\makeatother

\title{A Bers type classification of big mapping classes}

\author{Ara Basmajian}
\address[Ara Basmajian]{The Graduate Center, CUNY \\ 365 Fifth Ave., N.Y., N.Y., 10016, and   Hunter College, CUNY, 695 Park ave., N.Y.,N.Y., 10065, USA} 
 \email{abasmajian@gc.cuny.edu}
 \thanks{AB  supported by PSC CUNY Award 65245-00 53
and partially supported by  Simons  Collaboration Grant (359956, A.B.)}
 \author{Yassin Chandran}
\address[Yassin Chandran]{The Graduate Center, CUNY \\ 365 Fifth Ave., N.Y., N.Y., 10016, USA} 
\email{ychandran@gradcenter.cuny.edu}
\thanks{YC supported by  a CUNY pre-dissertation fellowship in 2022}
\begin{document}
\begin{abstract}  For an infinite type surface $\Sigma$, we consider the space of (marked) convex  hyperbolic structures  on $\Sigma$, 
denoted $H(\Sigma)$, with the Fenchel-Nielsen topology.  The (big) mapping class group  acts faithfully on this space
allowing us to  investigate a number of  mapping class group invariant subspaces of $H(\Sigma)$ which arise from various geometric properties 
(e.g. geodesic or metric completeness, ergodicity of the geodesic flow, lower systole bound, discrete length spectrum) of the hyperbolic structure. In particular, we  show  that the space of    geodesically complete  convex hyperbolic structures in $H(\Sigma)$    is  locally path connected, connected and decomposes naturally into Teichm\"uller subspaces.  The big mapping class  group of $\Sigma$ acts faithfully on this space  allowing us to  classify mapping classes into three types ({\it always quasiconformal, sometimes quasiconformal, and never quasiconformal})   in terms  of their   dynamics  on the Teichm\"uller subspaces. Moreover,  each  type contains  infinitely many mapping classes, and  the type is relative  to the underlying subspace of  $H(\Sigma)$ that is being considered. 

As an application of our work, we show that if the mapping class group of a general topological surface  $\Sigma$ is  algebraically isomorphic to the   modular  group of a  Riemann  surface
$X$,  then $\Sigma$ is of finite topological type and $X$ is homeomorphic to it.  Moreover, a big mapping class group can not act on any Teichm\"uller
   	space with orbits equivalent to modular group orbits.
\end{abstract}
\maketitle

%%%%%%%Introduction%%%%%%%%%%%%%%%%

\section{Introduction and results}

For finite   type surfaces, Teichm\"uller space  is the  space of marked hyperbolic structures  for which the mapping class group acts with the quotient  moduli space.   It  has been very successfully used to study deformations of hyperbolic structures, 3-manifolds, mapping class groups, as well as many other phenomena. Bers studied the action of $\Map(\Sigma)$ on $\mathcal{T}(\Sigma)$, and gave a different  proof of the celebrated Nielsen-Thurston classification using quasiconformal homeomorphisms  \cite{Bers}. For infinite type surfaces, there is no canonical choice of component of Teichm\"uller space. Indeed, there are uncountably many components that are fundamentally different  (see for example \cite{Ba1}, \cite{BaKim}, \cite{Ma2}). As a result, many earlier works have focussed  on a particular component by fixing a conformal structure, and the  
 action of the modular group (quasiconformal mapping class group). 
 
 In this work we take an expanded view by considering convex  hyperbolic structures on $\Sigma$. A {\it convex hyperbolic surface}  $X$ is  by definition  the interior of the convex core of  a geodesically complete hyperbolic surface. 
Examples of convex hyperbolic surfaces are illustrated in  Figures \ref{fig: Flute minus half-planes} and   \ref{fig: BiFlute}.
In Figure \ref{fig: Flute minus half-planes}  the convex hyperbolic surface 
(that is, the interior of its convex core) is a tight flute surface with half-planes deleted. In Figure  \ref{fig: BiFlute} the convex hyperbolic surface is the bi-infinite flute with its funnels deleted. In both figures the convex core is the red shaded region. 
 A {\it (marked) convex hyperbolic structure}  on $\Sigma$ is a  homeomorphism 
$f: \Sigma \rightarrow X$, where $X$  is a convex hyperbolic surface with the usual equivalence of two marked structures.  Such an equivalence class we usually denote $(X,f)$, and   the space of all convex   hyperbolic structures  with the Fenchel-Nielsen (product) topology is denoted
$H(\Sigma)$,  allowing us to set the groundwork for studying this space and  the associated effective action of the mapping class group on it.  Reminiscent of Bers approach to investigating  a mapping class we  consider to  what extent a big mapping class  has  a quasiconformal representative and use this to derive a classification of big mapping class elements.   
See Sections \ref{sec: Preliminaries} and \ref{sec: property $P$  subspaces} for definitions and details.

 %%%%%%%%%FIGURE%%%%%%%%%%%%%%%%
\begin{figure}[htb]\label{fig: Flute minus half-planes}
 \begin{center}
     \includegraphics[width=5.5in]{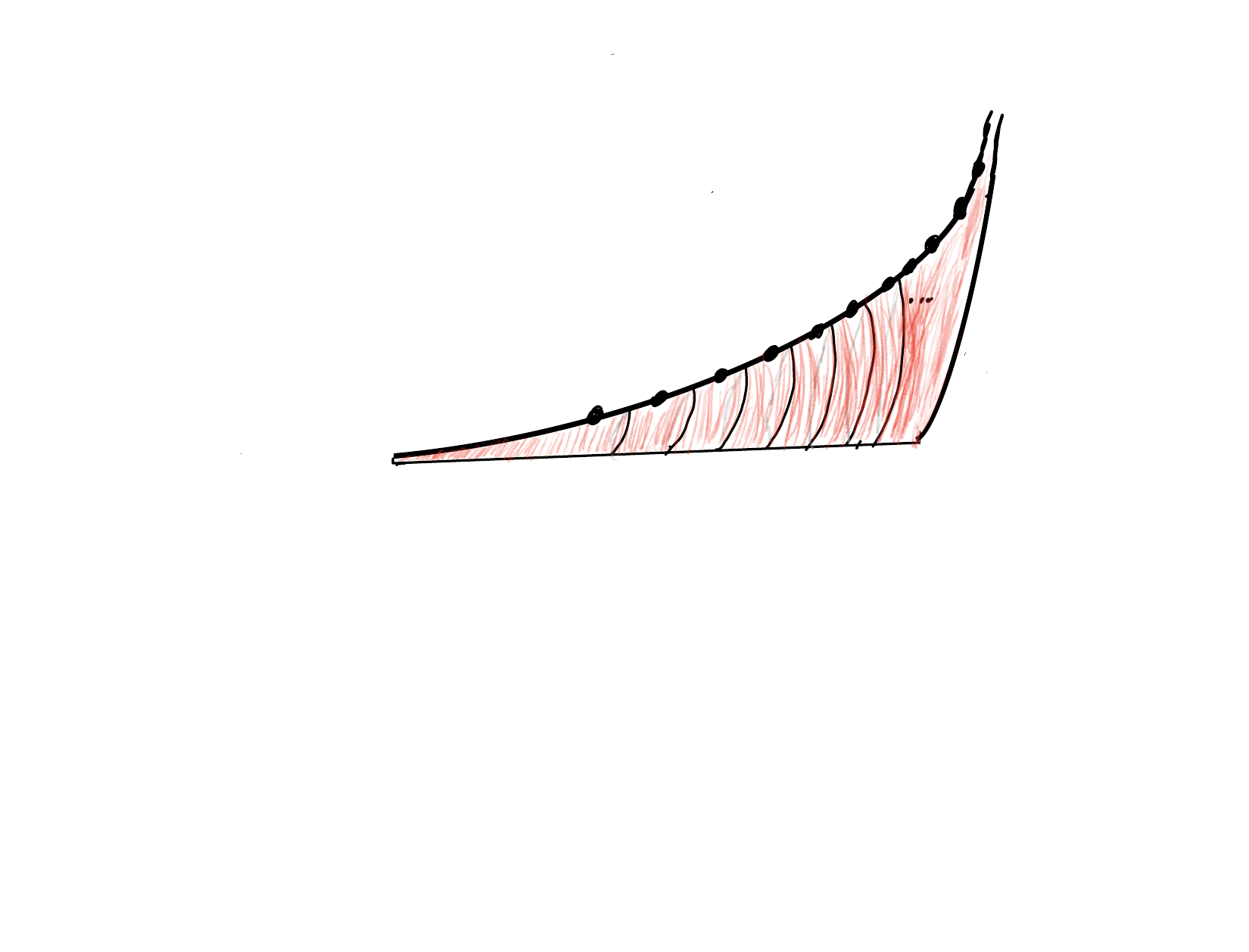}
     \vspace{-150pt}
    \caption{Tight flute minus half-planes}
    %\label{fig:setting}
  \end{center}
\end{figure}
%%%%%%%%%%%Figure%%%%%%%%%%%%

 %%%%%%%%%FIGURE%%%%%%%%%%%%%%%%
\begin{figure}[htb] 
 \begin{center}
     \includegraphics[width=5.5in]{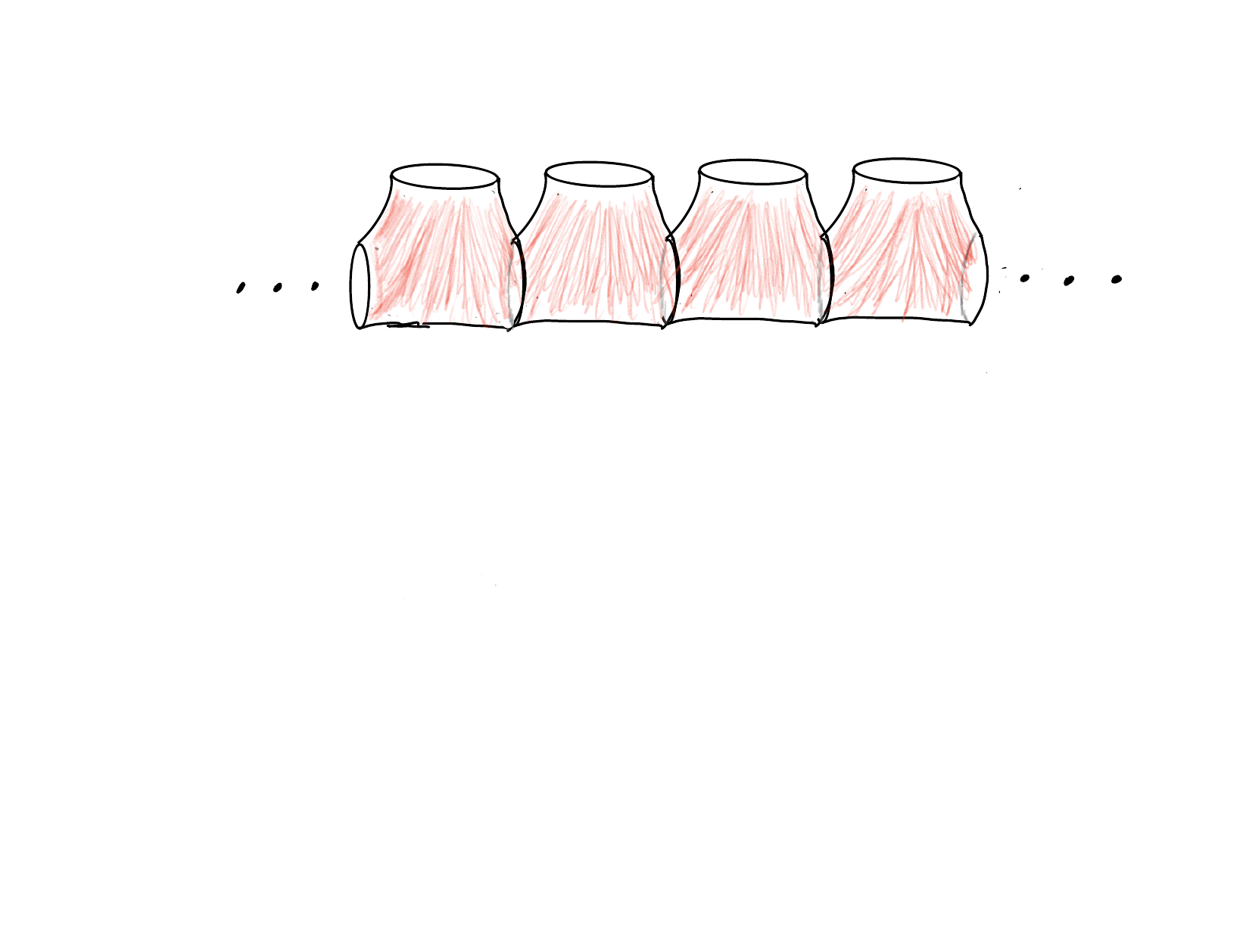}
     \vspace{-200pt}
    \caption{Bi-infinite flute minus funnels}
    %\label{fig:setting}
    \label{fig: BiFlute}
  \end{center}
\end{figure}
%%%%%%%%%%%Figure%%%%%%%%%%%%

It is  a consequence of \cite{BaKim} and \cite{BaSa} that the Fenchel-Nielsen length coordinates even for   complete hyperbolic surfaces can be wildly different. For example, the length coordinates can grow  super exponentially fast  or  super exponentially small.  One motivating question (due to Hugo Parlier) is the following, 
\vskip10pt
\noindent\textbf{Question:}	Does there exist   a continuous deformation between any two geodesically complete convex  hyperbolic structures through geodesically complete convex  hyperbolic structures. \vskip10pt

We note that F-N coordinates impose a straight line geometry on
$H(\Sigma)$. However  straight line paths between two geodesically complete convex structures do  not necessarily remain in the subspace of geodesically complete convex structures in $H(\Sigma)$.
  See Theorem \ref{thm: non-convexity}.  

 Let $P$ be a   quasiconformally invariant  property. We define
  the following subspaces

$$
\mathcal{C}_{P}=
\{(X,h) \in H(\Sigma): X \text{ has property P, all isolated planar ends are cusps}\}.
$$
$$
\mathcal{C}=
\{(X,h) \in H(\Sigma): \text{X is geodesically complete}\}.
$$
$$
\mathcal{C}^{\prime}=
\{(X,h) \in H(\Sigma): \text{X is metrically complete}\}.
$$
In the above,  $X$ is {\it geodesically complete}  if every geodesic extends infinitely in both directions.  $X$ is {\it metrically complete}  if its geodesic completion does not require adding $\text{half-planes}$.  Every convex hyperbolic structure has a unique extension to a geodesically complete structure. 

The subspace $\mathcal{C}$ is in fact  the set of $(X,f)$ where the isolated planar ends of $X$ are cusps and  for  any other end there is a  sequence of simple closed geodesics whose distances from any point in $X$ go to $\infty$.  Note  that being a geodesically complete hyperbolic structure or a metrically complete  hyperbolic structure is a quasiconformally  invariant property (see section 
\ref{sec: property $P$  subspaces} for details).

 \medskip\noindent\textbf{Theorem A.}[Theorem  \ref{thm: property P connectivity}]
 \textit{The subspace
   $\mathcal{C}_{P} \subset H(\Sigma)$ is   locally path connected,  connected and the subspace of metrically complete structures $\mathcal{C}^{\prime} \subset H(\Sigma)$ is  path connected. 
 In particular, the subspace of geodesically complete structures 
  $\mathcal{C} \subset H(\Sigma)$ is locally path connected, connected. }
\medskip

A {\it Teichm\"uller subspace}  is the set of quasiconformal deformations of a fixed marked structure.  The {\it Teichm\"uller space corresponding to 
the Riemann surface $X$},  denoted $\mathcal{T}(X)$, is the set of equivalence classes of quasiconformal deformations of $X$. We note here  that Teichm\"uller subspaces are defined for all convex structures,  whereas  by our convention a Teichm\"uller  space is only defined for a geodesically complete convex hyperbolic structure. See section  \ref{sec: QC homeomorphisms and Teich  subspaces} for the definition and  properties.

 \medskip\noindent\textbf{Corollary B.}[Corollary \ref{cor: disjoint union of Teichspaces}]
\textit{$H(\Sigma)$  as well as  the subspaces $\mathcal{C}_{P}$,  $\mathcal{C}$, 
 and $\mathcal{C}^{\prime}$   are partitioned into the disjoint union of  Teichm\"uller subspaces. In addition,  if  a Teichm\"uller subspace is comprised of geodesically complete structures    then there is an injective  continuous  mapping of the  Teichm\"uller space to the Teichm\"uller subspace. In particular, $\mathcal{C}$ is partitioned into the disjoint union of 
 Teichm\"uller spaces- one for each qc-type of Riemann surface structure on 
 $\Sigma$.}
\medskip

The {\it mapping class group of} $\Sigma$ is  
$$\Map(\Sigma) := \text{Homeo}^{+}(\Sigma)/\text{Homeo}_{0}(\Sigma)$$
where $\text{Homeo}^{+}(\Sigma)$ is the group of orientation preserving homeomorphisms of $\Sigma$,  and $\text{Homeo}_{0}(\Sigma)$ the normal subgroup of homeomorphisms isotopic to the identity.  
See \cite{Nick}.

 A mapping class $\phi$  is {\it quasiconformal}  if   there is a convex  hyperbolic structure $(X,f)$  and a quasiconformal mapping 
	$q: X \rightarrow X$ so that  the following diagram commutes up to isotopy 
	
\[ \begin{tikzcd} \label{dia: commuting diagram}
	\Sigma \arrow{r}{\phi} \arrow[swap]{d}{f} & \Sigma \arrow{d}{f} \\%
	X \arrow{r}{q}& X
\end{tikzcd}
\]

\subsection*{Definition:}
%%%%%

	Let  $\mathcal{D} \subseteq  H(\Sigma)$ be a subspace which is path connected and invariant under the mapping class group.  Then for $\phi \in \Map(\Sigma)$ we say that $\phi$ relative to $\mathcal{D}$ is
	\begin{itemize}
	\item  {\it always quasiconformal} if for any convex  hyperbolic structure $(X,f) \in \mathcal{D}$, $\phi$ admits a quasiconformal representative.
	\item {\it sometimes quasiconformal} if there exists  $(X,f), (Y,g) \in \mathcal{D}$ such that $\phi$ admits a quasiconformal representative with respect to $(X,f)$ but not $(Y,g)$.
	\item {\it never quasiconformal} if there does not exist  $(X,f) \in \mathcal{D}$ for which $\phi$ admits a quasiconformal representative. 
	\end{itemize}

Note that the three classes  are mutually exclusive and each $\phi \in \Map(\Sigma)$ falls into exactly one of the classes. We sometimes refer to this trichotomy as the three possible types of a mapping class. 
 We are mostly interested in the classification when the subspace 
 $\mathcal{D}$ is  either  $\mathcal{C}_{P}$, $\mathcal{C}$,  $\mathcal{C}^{\prime}$  or $H(\Sigma)$ itself. The  trichotomy is invariant under conjugation and therefore descends to any quotient of the mapping class group.

\subsection*{Action of $\Map$ on $H(\Sigma)$:}
%%%%%

The mapping class group, denoted  $\Map(\Sigma)$, acts on $H(\Sigma)$  by precomposition of the marking map, 
$$\phi \cdot (X,f) := (X, f \circ \phi^{-1}).$$

We have the following theorem

 \medskip\noindent\textbf{Theorem C.}[Theorem \ref{thm: type classification}]
 \textit{ $\Map(\Sigma)$ acts faithfully as a group of homeomorphisms of 
$H(\Sigma)$,  keeps invariant $\mathcal{C}_{P} \subset H(\Sigma)$, and 
permutes  the set of Teichm\"uller subspaces of $\mathcal{C}_{P}$. 
Moreover, for $\phi \in \Map(\Sigma)$ }
\begin{itemize}
\item \textit{$\phi$ is always quasiconformal rel. $\mathcal{C}_{P}$  if and only if  $\phi$ keeps invariant  every Teichm\"uller subspace in $\mathcal{C}_{P}$  if and only if $\phi$ is finitely  supported. }

\item \textit{$\phi$ is sometimes  quasiconformal rel. $\mathcal{C}_{P}$ if and only if there exists at least one Teichm\"uller subspace  in $\mathcal{C}_{P}$  that is kept invariant and at least one in $\mathcal{C}_{P}$ that is not kept invariant. }

\item \textit{$\phi$ is never quasiconformal rel. $\mathcal{C}_{P}$ if and only if   there is no Teichm\"uller subspace in  $\mathcal{C}_{P}$ kept invariant by 
 $\phi$.}
\end{itemize}
\medskip

We remark that the almost qc and never qc  types  contain   uncountably many conjugacy classes  of elements. 
Though the above theorem is phrased in terms of the action of $\Map(\Sigma)$ on $\mathcal{C}_{P}$, in fact the same statement is true with 
$\mathcal{C}_{P}$ replaced by $H(\Sigma)$ and  finitely supported  replaced by compactly supported.  Since being geodesically complete is a quasi-invariant property  the above theorem applies to the subspace of  geodesically complete structures,  $\mathcal{C}$.

A feature of this classification is that the type of the element is relative to the subspace the mapping class group acts on. 
 In Section \ref{sec: rel. trichotomies}, we illustrate this with  a number of examples, which  in fact show how the context of which subspace $\Map$ acts  influences the analytic characterization  of the mapping class. 
 
We denote  the {\it quasiconformal mapping class group} of a Riemann surface $X$
 by $\Mod(X)$. As  an application  of our work we show that for any infinite type surface $\Sigma$ and for $X$ any hyperbolic surface, there does not exist a group  isomorphism $\rho: \Map(\Sigma) \to \Mod(X)$. We emphasize that $\Sigma$ and $X$ need not be related, allowing  us to conclude that there 
   is no natural action of a big mapping class group on \textit{any} Teichm\"uller space.
 
   \medskip\noindent\textbf{Corollary D.}[Corollary \ref{cor: MCG does not act}]
   \textit{   Suppose $\Sigma$ is a  general topological surface of negative Euler characteristic.  If $\Map(\Sigma)$ is isomorphic to $\Mod(X)$ for some Riemann surface
   	$X$, then $\Sigma$ has finite topological type and X is homeomorphic to $\Sigma$.  Moreover, a big mapping class group can not act on any Teichm\"uller
   	space with orbits equivalent to modular group orbits.  }
   \medskip

\subsection{Some background:}

Examples of infinite type surfaces include, the  tight flute surface, the  Loch Ness Monster,
or the  Cantor Surface  (See Figure  \ref{fig:Surfaces}).

%%%%%%%%%FIGURE%%%%%%%%%%%%%%%%
\begin{figure}[htb]\label{fig:Surfaces}
 \begin{center}
     \includegraphics[width=5.5in]{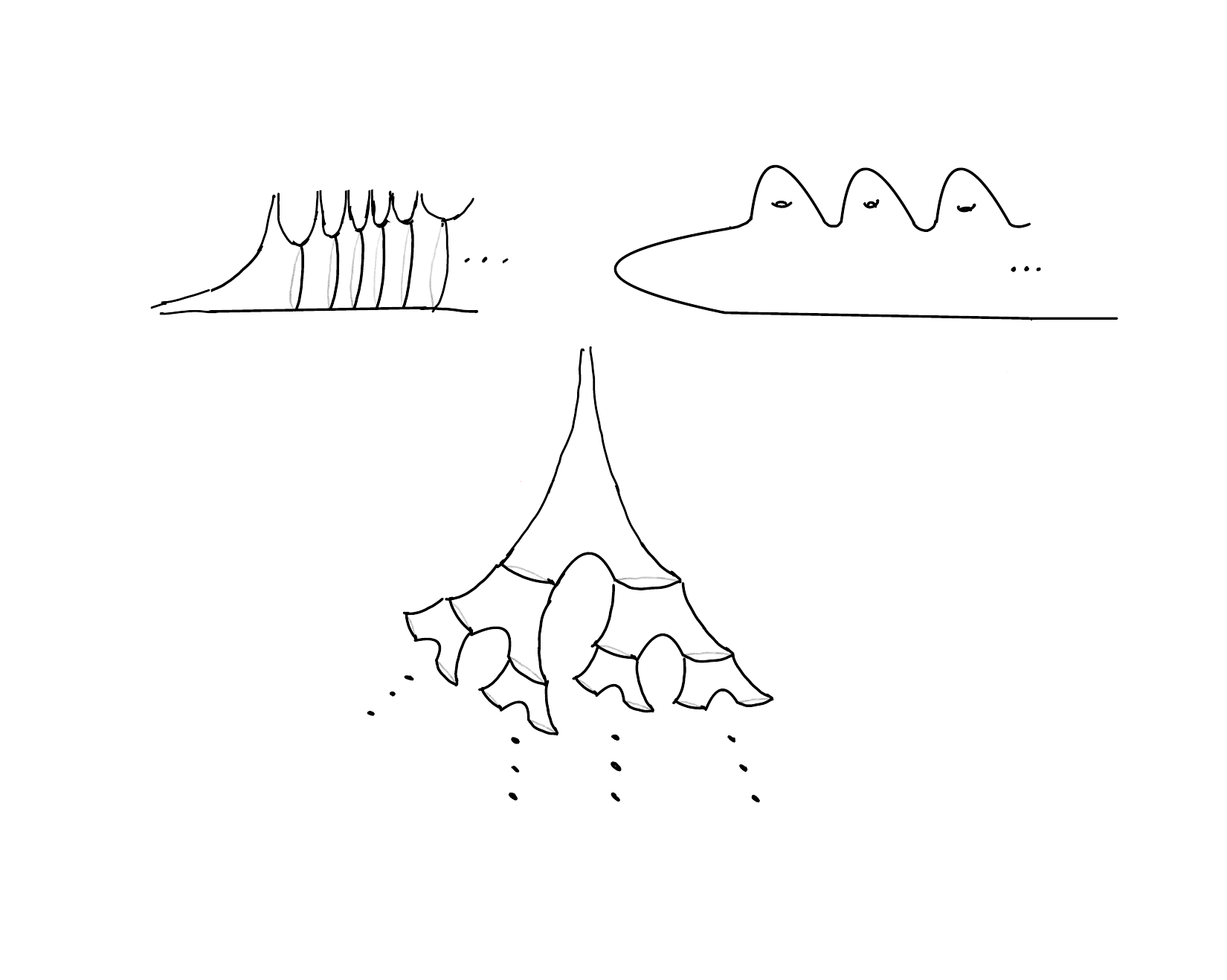}
     \vspace{-20pt}
    \caption{Flute, Loch Ness Monster, and Cantor tree Surface}
    %\label{fig:setting}
  \end{center}
\end{figure}
%%%%%%%%%%%Figure%%%%%%%%%%%%

In this paper, we introduce the  space of (marked) convex  hyperbolic structures associated to $\Sigma$, denoted  $H(\Sigma)$,  in order to study simultaneously all hyperbolic structures on $\Sigma$ as well as organize big mapping classes into a trichotomy according to their action on the space. We also use this space to study continuous deformations of various hyperbolic structures (Theorem \ref{thm: property P connectivity}).
 
$H(\Sigma)$ by definition is completely metrizable and  has the  property that
 for any  convex hyperbolic structure $(X,f)$ there is an associated 
 {\it Teichm\"uller subspace} through $(X,f)$, denoted  $\mathcal{T} (X,f)$. 
 Moreover, if the marked hyperbolic structure is complete then there 
 is a continuous injection of $\mathcal{T}(X) \hookrightarrow H(\Sigma)$ whose image is $\mathcal{T} (X,f)$.

For  $P$ a quasiconformally invariant property, the possible subspaces 
$\mathcal{C}_{P}$ are rich and quite varied. 
Some examples of historically studied quasi-invariant properties include: geodesic completeness (\cite{Ba2} \cite{BaSa}),  ergodicity of the  geodesic flow (equivalently not carrying a Green's function) \cite{BHS},  discreteness of the  length spectrum  \cite{BaKim}, admitting a bounded pants decomposition (\cite{Shi} \cite{Kin}),  
having a countable quasiconformal mapping class group \cite{Ma2},  bounded injectivity radius \cite{BPV}, and    being of the first kind. Quasi-invariant properties of the length spectrum are also a source of possible subspaces of hyperbolic structures. 

An important tool we use in proving the connectivity results is the following

 \medskip\noindent\textbf{Zig-Zag. }[Lemma \ref{lem: zig-zag path}]
\textit{Given   $(X,f)$ and  $(Y, g)$ in  $\mathcal{C}_{P} \subset H(\Sigma)$, there exists a  zig-zag path $L:[0,1] \to  \mathcal{C}_{P}$ from  $(X,f)$ to $(Y,g)$,
where 
	$L(t)$ for any $t \in [0,1)$ is in the same Teichm\"uller subspace as $L(0) =(X,f)$.}
\medskip

\noindent{{\bf Organization of the sections:}} 
The  sections in the sequel are organized in the following way. In  section 
\ref{sec: Preliminaries} we set  preliminaries and notation.  In section 
 \ref{sec: property $P$  subspaces} we discuss end geometry. 
In section \ref{sec: QC homeomorphisms and Teich  subspaces} we define Teichm\"uller subspaces of the space of hyperbolic structures,  subspaces  of hyperbolic structures that have property $P$,  and 
prove connectivity properties. In section 
\ref{sec: The space of complete structures is not convex} we show that 
the space of geodesically complete structures is not convex.
In section  \ref{sec: MCG action}, we introduce the action of the mapping class group on the space of hyperbolic structures.
In section  \ref{sec: trichotomy} we introduce the trichotomy for big mapping class groups and prove that always 
$qc$-mappings are compactly supported. 
In section  \ref{sec: Constructing never qc maps}, we construct never 
$qc$ mappings. In section \ref{sec: MCG action 
on Teichmuller spaces} we show that a big mapping class group is algebraically distinct from any modular group and does not act naturally on any Teichm\"uller space.  
Finally, in section \ref{sec: rel. trichotomies}, we investigate 
the type of a mapping class  relative to a subspace. Among the examples we construct is an infinite nested set of subspaces in $H(\Sigma)$ for which each pair of adjacent subspaces has a mapping class that changes type when going from one to the other.

{\noindent{\bf Acknowledgement:}}  The authors would like to thank Chris Bishop, Hugo Parlier, Dragomir \v{S}ari\'{c},  Ser Peow Tan, Perry Susskind, and Nick Vlamis  for helpful conversations. Parlier asked   the initial question of connectivity.

%%%%%%%Definitions and basics%%%%%%%%%%

\section{Preliminaries}
\label{sec: Preliminaries}

\subsection{Basics and notation}
As basic references we refer to  the books,
 Buser \cite{Bus} for hyperbolic geometry basics and   Farb-Margalit  \cite{F-M} for both hyperbolic geometry and the basics on the mapping class group. For the theory of quasiconformal mappings we refer to 
 \cite{Ahl, Vis},   and for Teichm\"uller theory \cite{F-M, Fl-Ma}.

%%%%%%%%%%%%Table: Notation%%%%%%%%%%%%%%%%%%%%%

\begin{center}
	\begin{tabular}{| |l  c |  }
		\hline \hline 
		{\bf Definition, Notation} & {\bf Section}  \\
		\hline
		\hline 
		
		\hline
		Topological surface, $\Sigma$ &  \\
		\hline

		Fenchel-Nielsen coordinates, $\vec{z} = (z_1, z_2, \cdots)$  & 
		\ref{sec: Preliminaries}  \\
		\hline
		(Marked) convex hyperbolic structure, $(X,f)$ &  2  \\
		\hline
		Geodesically complete  convex structure & \ref{sec: Preliminaries}    \\
		\hline
		Metrically  complete  convex structure & \ref{sec: Preliminaries}    \\
		\hline
		Convex core of $X$ & \ref{sec: Preliminaries}   \\
		\hline
		End geometries on $\Sigma$  &\ref{sec: property $P$  subspaces}   \\
		\hline
		Space of marked (convex) hyperbolic structures, $H(\Sigma)$  & 
		\ref{sec: QC homeomorphisms and Teich  subspaces}  \\
		\hline 
		Space of geodesically complete structures,  $\mathcal{C}$ & 
		\ref{sec: QC homeomorphisms and Teich  subspaces}  \\
		\hline 
		Space of metrically complete structures, $\mathcal{C}^{\prime}$ & \ref{sec: QC homeomorphisms and Teich  subspaces} \\
		\hline
		Space of convex structures with property $P$, 	$\mathcal{C}_{P} \subset H(\Sigma)$  & \ref{sec: QC homeomorphisms and Teich  subspaces}  
	   \\

		\hline
		Teichm\"uller space of $X$, $\mathcal{T}(X)$ & 
		\ref{sec: QC homeomorphisms and Teich  subspaces}    \\
		\hline
		Teichm\"uller subspace through $(X,h)$,  $\mathcal{T} (X,h)$  & 
		\ref{sec: QC homeomorphisms and Teich  subspaces}  \\
				\hline

		Mapping class group of $\Sigma$, $\Map=\Map(\Sigma)$ & 
		\ref{sec: MCG action}  \\
		\hline
		Modular group of the Riemann surface $X$, $\text{Mod}(X)$& 
		\ref{sec: MCG action on Teichmuller spaces}
 \\
		\hline
		
		\hline

	\end{tabular} 
\end{center}
\label{Table:notation}

\subsection{Hyperbolic structures}
Denoting the hyperbolic plane by $\mathbb{H}$, a  {\it geodesically complete hyperbolic surface}  is by definition  $\mathbb{H}/G$, where $G$ is a torsion-free discrete subgroup of orientation preserving isometries of 
$\mathbb{H}$. Denote the limit set of $G$ by 
$\Lambda (G)$. The {\it convex core}  of  $\mathbb{H}/G$ is the  projection of the convex hull of $\Lambda (G)$ to the quotient. It  is the smallest geodesically convex subspace of $\mathbb{H}/G$ which contains all the homotopy classes of curves. The complement of the convex core consists of  the (possibly empty) union of  funnels and half-planes (see \cite{Ba2} and \cite{BaSa}).  
The convex core uniquely determines its completion to a geodesically complete hyperbolic surface. 
A {\it convex  hyperbolic surface} $X$  is  by definition  the interior of the convex core of  a geodesically complete hyperbolic surface
(see Figures  \ref{fig: Flute minus half-planes} and   \ref{fig: BiFlute}).  $X$  is said to be 
{\it geodesically complete}  if every geodesic can be extended infinitely far in both directions. Equivalently,  $X$ is equal to its geodesic completion. 
The convex hyperbolic surface $X$ is said to be {\it metrically complete} if it satisfies one of the following equivalent conditions:

\begin{itemize}
\item  the  geodesic completion of $X$ consists  of  adding a (possibly empty) union of funnels 
\item the   geodesic completion of $X$ does not require the addition of 
half-planes 
\item  the boundary components of the metric closure of $X$  are compact. 
\end{itemize}

Fix $\Sigma$ an infinite-type topological (orientable) surface.
A {\it (marked) convex  hyperbolic structure}  on $\Sigma$ is an orientation preserving homeomorphism 
$f: \Sigma \rightarrow X$, where $X$  is a convex hyperbolic surface. 
Two  such  homeomorphisms 
$f_1:\Sigma \rightarrow X_2$,  and  $f_2:\Sigma \rightarrow X_2$ are said  to be equivalent (denoted $f_1 \sim f_2$),  if there exists an isometric  mapping 
$c:X_1 \rightarrow X_2$ so that the following diagram commutes up to homotopy,

\begin{equation}\label{diagram: commuting up to homotopy}
	\begin{tikzcd} 
		\Sigma  \arrow{r}{f_1} \arrow[swap]{dr}{f_2} & X_1 \arrow{d}{c} \\
		& X_2
	\end{tikzcd}
\end{equation}

We denote the equivalence class of $f :\Sigma \rightarrow X$ by $(X,f)$.
  The  space of all (marked)  convex hyperbolic structures  on $\Sigma$ is 

$$H(\Sigma):= \{(X,f): X \text{ a convex hyperbolic surface} \}.$$

\begin{rem}
	If when defining $H(\Sigma)$ we did not require that $X$ is a convex hyperbolic surface and instead adopted the standard convention used for finite-type surfaces,  that $X = \mathbb{H}/ G$ for some discrete torsion-free subgroup $G$ of isometries of $\mathbb{H}$, then $H(\Sigma)$ would no longer contain all hyperbolic structures carried by $\Sigma$. For $\Gamma$ a pants decomposition of $\Sigma$ and  any choice of Fenchel-Nielsen coordinates determining  $X$ so that $X-C(X)$ contains half-planes, the marking map fails to be a homeomorphism. To see this must be so, note that  $\Sigma - \Gamma$ is a union of pairs of pants   but $X -  f(\Gamma)$ contains half-planes.
\end{rem}

\noindent We say that $(X,f)$ is a {\it geodesically complete structure} if $X$ is a geodesically complete  convex hyperbolic surface, and that $(X,f)$ is  a {\it metrically complete structure}  if $X$ is a metrically complete  convex hyperbolic surface.

For $\gamma$ either a  closed curve or an arc between two simple closed curves on $\Sigma$ we define the length function 
$\ell_{\gamma} : H(\Sigma) \rightarrow \mathbb{R}_{\geq 0}$ of a  
marked structure $(X,f)$  to  be the length of the geodesic in the free homotopy class of  $f(\gamma)$. If $\gamma$ is an arc  between boundary  curves then the homotopy is relative to the boundary. A simple closed curve is said to be {\it essential} if it is not homotopic to a point or boundary parallel.
A {\it peripheral}  simple closed curve is one that bounds a puncture of 
$\Sigma$.

We will put {\it Fenchel-Nielsen  coordinates} (FN-coordinates)   on  $H(\Sigma)$. To that end,  suppose $\Gamma$  is a topological pants decomposition of $\Sigma$  with a choice of seams on each nonperipheral pants curve on 
$\Sigma$, and let $\{\gamma_i \}$ be  an ordered list of the pants curves.  We assign to the convex hyperbolic structure   $(X,f)  \in H(\Sigma)$  the FN-coordinates
$(\ell_{f(\gamma_i)}(X),  t_{f(\gamma_i)}(X))$
endowed    with the product topology thus making  $H(\Sigma)$ homeomorphic to 
$(\mathbb{R_{+}}\times \mathbb{R})^{\infty}$. We remark here that 
the image of the pants curves when straightened are not necessarily a geometric pants decomposition for the geodesic completion of $X$. Note that for peripheral coordinates we stipulate that the twist parameter is equal to zero.
As this is a countable product of the real numbers it is  metrizable. For convenience we  use the following metric for 
FN-coordinates which induces the product topology.
For $(\ell_n, t_n), (\ell^{\prime}_n,t^{\prime}_n)  \in  (\mathbb{R_{+}}\times \mathbb{R})^{\infty}$,

\begin{equation}
	d((\ell_n,t_n),(\ell^{\prime}_n,t^{\prime}_n))=
	\sum_{i=1}^{\infty}  \frac{1}{2^i} \left[\frac{ | \ell_n -\ell_n^{\prime}|}{(1+  |\ell_n -\ell_n^{\prime}|)}+
	\frac{ |t_n - t_n^{\prime}|}{(1+  |t_n -t_n^{\prime}|)}\right].
\end{equation}

To ease notation,  we sometimes use complex coordinates. Let 
$\text{U} \subset \mathbb{C}$ denote the upper half-plane. Then for points in FN-coordinates we sometimes use the notation,
$\vec{z}=(z_1,z_2,z_3,...)$ for a point in the infinite product $\text{U}^{\infty}$.

We emphasize that in  this paper we exclusively use the product topology on the Fenchel-Nielsen coordinates. For previous works considering the product topology on such a space see   \cite{Ba2},  \cite{Ba1}, and \cite{BaKim}.

\begin{lem}
	The topology on (any subspace of) $H(\Sigma)$ is independent of the choice of pants decomposition. In particular, given two pants decompositions $\Gamma_1, \Gamma_2$, and their induced coordinate spaces $H_1, H_2$ the identity map on the space of marked hyperbolic structures induces a homeomorphism between the coordinate spaces.  If $\gamma$ is the homotopy class of a closed curve then the  length function 
	$\ell_{\gamma} : H(\Sigma) \rightarrow \mathbb{R}_{\geq 0}$ is a continuous function. 
\end{lem}

\begin{proof}
	Let $\Gamma_1$ and $\Gamma_2$ be two pants decompositions of $\Sigma$, and let $H_1, H_2$ be the corresponding coordinate spaces each homeomorphic to $\prod (\mathbb{R}_{>0}, \mathbb{R})$. Identify each coordinate with the upper half-plane. A sequence in $H_1$ converges if each coordinate converges. If $(X_i, f_i)$ converges to $(X,f)$ in $H_1$ then each coordinate $z_{\gamma, i}$ converges to $z_\gamma$ for each $\gamma \in \Gamma_1$. Each $\alpha \in \Gamma_2$ is contained in a finite subsurface $S$ whose boundary is a finite collection of curves in $\Gamma_1$. Since the Fenchel-Nielsen coordinates on $S$ converge, the corresponding length, and twist parameter $t_\alpha$ also converges.  
	
	The fact that $\ell_{\gamma}$ is continuous follows from the fact that 
	$\gamma$ is contained in topologically finite subsurface $\Sigma^{\prime} \subset \Sigma$, and hence   $\ell_{\gamma}$ depends on only finitely many Fenchel-Nielsen coordinates. Continuity now follows 
	(see \cite{F-M}).
\end{proof}

%%%%%%%%%%%%%End geometry vs. end topology%%%

\section{End geometry of a hyperbolic surface}
\label{sec: property $P$  subspaces}

Denote the end topology of an orientable topological surface $\Sigma$ by $\mathcal{E}(\Sigma)$.
We give a brief description here of the space of ends. For more details on ends  and other terminology we refer the reader to \cite{BaSa}.  Let 
$\{\Sigma_n \}$  be a compact exhaustion of $\Sigma$. An {\it end} of 
$\Sigma$ is a decreasing sequence  of complementary connected components of 
$\{\Sigma-\Sigma_n \}$ given by 
$$
\{\mathcal{U}_{1} \supset \mathcal{U}_{2}  \supset \mathcal{U}_{3} \supset ...\}.
$$
The {\it space of ends}  inherits a natural topology from the topology of $\Sigma$ which does not depend on the choice of compact exhaustion.  
  A {\it non-planar end}  is an end for which $\mathcal{U}_{n}$ has infinite genus for all $n$. Otherwise, the end is said to be planar. The non-planar ends form a closed subspace of the space of ends. 
 It is well-known (\cite{Ri}) that the genus and the double topological space (non-planar ends, ends) determines the topology of $\Sigma$. 
Moreover, a  homeomorphism  between topological surfaces 
$f: \Sigma_1 \rightarrow \Sigma_2$ induces a homeomorphism $f_{\ast}: \mathcal{E}(\Sigma_1)\rightarrow \mathcal{E}(\Sigma_2)$.

Let  $e=\{\mathcal{U}_{1} \supset \mathcal{U}_{2}  \supset \mathcal{U}_{3} \supset ...\}$  be an end of  the topological  surface $\Sigma$, and
 $(X,f) \in H(\Sigma)$. Recalling that $X$ is a convex hyperbolic surface,  the possible geometries  of the end  $f_{*}(e)$ are either  
 \begin{itemize}
 \item  a cusp  if  $f(\mathcal{U}_{i})$ for some $i$   is conformal to  a punctured  disc in  $X$
 \item a funnel   if $f(\mathcal{U}_{i})$ for some $i$   is contained in a collar neighborhood about a compact  boundary component of the metric closure   of $X$
 \item for every  sequence of simple closed curves  $\{\gamma_{i}\}$ in 
   $\Sigma$ for which  $\gamma_{i} \rightarrow e$,  the straightened
   $X$-geodesics $\{f(\gamma_i)\}$ leave every compact subset of $X$
 \item there exists a   sequence of simple closed curves  $\{\gamma_{i}\}$ in 
   $\Sigma$ for which  $\gamma_{i} \rightarrow e$,  but  infinitely many of the straightened
   $X$-geodesics $\{f(\gamma_i)\}$  are a finite distance from any point in $X$. 
 \end{itemize}
 
 Thus the  geometry associated to an isolated planar end is either a cusp or a funnel. For  the other possible topological ends (isolated non-planar or not isolated) the ending geometry will be either a sequence of simple closed geodesics  or the disjoint union of  (up to a countable number)  half-planes. 
If the hyperbolic surface is   metrically  complete   then the geometry of any end   is   either  a cusp, a  funnel, or a sequence of simple closed geodesics that leave every compact set. If the hyperbolic surface   is geodesically complete then  the geometry of an  end is  either a cusp or a  sequence of simple closed  geodesics that leave very compact set. (See \cite{Ba2} and \cite{BaSa}) for details.

 %%%%%%%%%%%%%%%%%%%%%%%%%%%%%%%%%

\section{QC homeomorphisms, Teichm\"uller  subspaces, and Zig-zag paths}
\label{sec: QC homeomorphisms and Teich  subspaces}

In this section we bring quasiconformal  (qc) mappings  and 
Teichm\"uller space into the picture. We begin with a Lemma.

\begin{lem} \label{lem: change an FN-coordinate}
	Let $\{\gamma_i \}$  be the pants curves of a  pants decomposition  of $\Sigma$. Suppose 
	$f: \Sigma \rightarrow X$ and $g: \Sigma \rightarrow Y$ are marked convex hyperbolic structures  in $H(\Sigma)$
	having  identical  FN-coordinates for all $\{\gamma_i\}$ except one coordinate which is nonzero in both structures.  Then there exists a qc-mapping $q: X \rightarrow Y$.
\end{lem}

\begin{proof} Let $\gamma_n$ be the pants curve for which the marked hyperbolic structure length and twist coordinates differ. That is,
	$\ell_{\gamma_n} (X) \neq  \ell_{\gamma_n} (Y)$ and
	$t_{\gamma_n} (X) \neq  t_{\gamma_n} (Y)$.
	Use Bishop's theorem  (see \cite{Bish})  to get a qc mapping that fixes all the coordinates except the one for $\gamma_n$ for which it takes 
	$\ell_{\gamma_n} (X)$ to  $\ell_{\gamma_n} (Y)$. We next will  follow this map by the qc mapping which is a Nielsen twist along $\gamma_n$. To this  end we outline the steps. Assume   the $FN$-coordinates for $X$ and $Y$ differ only with the twists $t_{\gamma_n} (X)$ and $t_{\gamma_n} (Y)$. Set 
	$X=\mathbb{H}/G_{X}$ and $Y=\mathbb{H}/G_{Y}$. Lift this configuration to $\mathbb{H}$ and normalize the lift of $\gamma_n$ as the positive imaginary axis in the upper half-plane. There exists an equivariant quasi-isometry from 
	$\mathbb{H}$ to $\mathbb{H}$. Now the quasi-isometry induces an equivariant quasisymmetric mapping on 
	$\partial \mathbb{H}$, and  this quasisymmetric mapping extends to  an equivariant qc mapping of the upper half-plane. Finally, being equivariant,  this map descends to a qc mapping from $X$ to $Y$.
\end{proof}
The above lemma also follows from Lemma 8.2 in \cite{ALPS} and Matsuzaki's estimates on dilatations of multi-twists from \cite{Ma1}.

\subsection*{Teichm\"uller subspaces:} For each  $(X,h) \in H(\Sigma)$ we define the {\it Teichm\"uller subspace}, $\mathcal{T} (X,h)$, 
passing through $(X,h)$ to be 
\begin{displaymath}
\mathcal{T} (X,h) :=\{(Y,g) \in H(\Sigma): \exists \,
q:X \xrightarrow{qc} Y \text{where  diagram  \ref{diagram: qc, commuting up to homotopy} commutes up to homotopy}\}
\end{displaymath}

\begin{equation}\label{diagram: qc, commuting up to homotopy}
	\begin{tikzcd} 
		\Sigma  \arrow{r}{h} \arrow[swap]{dr}{g} & X  \arrow{d}{q} \\
		& Y
	\end{tikzcd}
\end{equation}

\begin{lem}  \label{lem: Teich subspaces disjoint}  Teichm\"uller subspaces  satisfy the following:
	
	\begin{enumerate}
		
		\item $\mathcal{T}(X,h)$  has no interior
		\item (Teichm\"uller subspaces are disjoint) 
		For $(X,h), (Y,g) \in H(\Sigma)$
\begin{displaymath}
\mathcal{T}(X,h) \cap \mathcal{T}(Y,g)	\ne \emptyset
	\Leftrightarrow 
		 \mathcal{T}(X,h)= \mathcal{T}(Y,g) \Leftrightarrow
	 g \circ h^{-1} \text{ is homotopic to a qc map}.
		\end{displaymath}
\end{enumerate}
\end{lem}

\begin{proof} 

 Item (1) follows from the fact that any open set in $H(\Sigma)$ is determined by restrictions on at most finitely many coordinates. Hence, arbitrary choices on the remaining (infinitely many) coordinates lead to surfaces that are not quasiconformally equivalent to $(X,h)$. For the proof of item  (2), if
$(Z,f) \in \mathcal{T}(X,h) \bigcap \mathcal{T}(Y,g)
		\ne \emptyset$ then,  by definition of  $\mathcal{T}(X,h)$ and 
		$\mathcal{T}(Y,g)$,   we have the commuting (up to isotopy)  diagram (\ref{eq: teich intersection}) where $q_1$ and $q_2$ are quasiconformal mappings.

\begin{equation}
\begin{tikzcd} \label{eq: teich intersection}
Y \arrow[leftarrow]{r}{g}\arrow[]{rd}{q_2}
                          &\Sigma  \arrow[]{d}{f} \arrow[]{r}{h}
&X \arrow[]{ld}{q_1}\\
&Z
\end{tikzcd}
\end{equation}

		 Following the outer triangle of diagram 
 (\ref{eq: teich intersection}) leads to the fact that 
 $(X,h) \in 	\mathcal{T}(Y,g)$. Similarly, $(Y,g) \in \mathcal{T}(X,h)$  and hence
 $ \mathcal{T}(X,h)= \mathcal{T}(Y,g)$. The last implication is also clear from 
 the outer triangle of the same diagram.   We have shown the forward implications of the  item (2) statement. It is straightforward  to check that these arguments are reversible  yielding the converse implications. We leave the details to the reader. 
\end{proof}

\subsection*{Geodesically complete structures and Teichm\"uller space:} 

Fix a Riemann surface $X=\mathbb{H}/G$ homeomorphic to $\Sigma$ where $G$ is of the first kind ($\Lambda (G)=\partial \mathbb{H}$). Equivalently the unique hyperbolic metric in the conformal class of $X$ has  geodesically complete convex core.  Let  $\mathcal{T}(X)$ be   the (quasiconformal) Teichm\"uller space of $X$ endowed with the Teichm\"uller metric; that is,
a point in Teichm\"uller space is a quasiconformal mapping $f: X \rightarrow Y$ with equivalence relation as in diagram 
(\ref{diagram: commuting up to homotopy})
with  the topology induced by the Teichm\"uller metric.  We denote the equivalence class of  $f: X \rightarrow Y$ by 
$[f: X \rightarrow Y]$, and  note that $\mathcal{T}(X)$
and $\mathcal{T}(Y)$ are identified if $X$ is conformally equivalent to $Y$. 
Associated to each marking $(X,h)$ of $X$ is a    map of $\mathcal{T}(X)$ into  $H(\Sigma)$ with the basepoint 
$[id:X \rightarrow X] \in \mathcal{T}(X)$ mapped to  
$(X,h) \in H(\Sigma)$.  
Namely,  each  $(X,h)  \in \text{H}(\Sigma)$ induces a  mapping,
$\Phi_h : \mathcal{T}(X) \rightarrow  H(\Sigma)$
given by 
\begin{equation}
	[f:X \rightarrow Y] \mapsto
	(Y, f \circ h).
\end{equation}

\begin{lem}  \label{lem: Image of Teich} Suppose  $(X,h) \in \mathcal{C}$. Then the  map $\Phi_h : \mathcal{T}(X) \rightarrow H(\Sigma)$ satisfies:
	
	\begin{enumerate}
	\item  The image of $\Phi_h$ is the Teichm\"uller subspace 
	$\mathcal{T} (X,h)$
		\item $\Phi_h$ is injective and continuous  
		\item 
		 $\mathcal{T}(X,h)$  is dense in 
		$\mathcal{C}$. 
		\item  $\Phi_h$ is  not a topological embedding
		
		\end{enumerate}
\end{lem}

\begin{proof}  
                 For   item (1), if $(Z,g)$ is in the image of $\Phi_h$ then 
                $(Z,g)$ is equivalent to $(Y, q \circ h)$ where  $q$ is a qc mapping from     $X$ to $Y$. Therefore there is an isometry $c: Y \rightarrow Z$ so that  up to isotopy  the following diagram commutes 
                
\begin{equation}\label{diagram: image of Teich space}
	\begin{tikzcd} 
		\Sigma  \arrow{r}{h} \arrow[swap]{dr}{g} & X  \arrow{d}{c \circ q} \\
		& Z
	\end{tikzcd}
\end{equation}
But this is exactly what it means for $(Z,g)$ to be in $\mathcal{T}(X,h)$.
 Hence, the image of $\Phi_h \subseteq \mathcal{T}(X,h)$. On the other hand,  if  $(Y,g) \in \mathcal{T}(X,h)$ then there is a 
                 quasiconformal mapping $q: X \rightarrow Y$ in  $\mathcal{T}(X)$.
                 Thus $\Phi_{h}$ surjects onto $\mathcal{T}(X,h)$.
                 
		For (2),  injectivity is clear. The continuity follows from the fact that hyperbolic length and twist parameters are quasi-invariant under a 
		$K$-quasiconformal mapping, and that $\Phi_h$ is continuous if and only if the coordinate functions $\ell_{\gamma}$ and  $t_{\gamma}$ are continuous for each $\gamma$ a pants curve.  
	
Item (3) follows from the fact that  there is a zig-zag path in  
$\mathcal{T}(X,h)$ that can get arbitrarily close to any 
$(X,f) \in \mathcal{C}$.
Finally  for item (4), the failure to be an embedding (that is, homeomorphic to the closure of its image) follows from item (3). 	
\end{proof}

As a consequence of  Lemmas  \ref{lem: Teich subspaces disjoint} and \ref{lem: Image of Teich},  the space $\text{H}(\Sigma)$ decomposes into the disjoint union of Teichm\"uller subspaces--one for each qc-type of Riemann surface structure on $\Sigma$. Furthermore, this decomposition is invariant under the action of  $\Map$ on $H(\Sigma)$.  Since Teichm\"uller subspaces can be identified with Teichm\"uller spaces for geodesically complete structures, we have  that  $\mathcal{C} \subset H(\Sigma)$  is the disjoint union of Teichm\"uller spaces- one for each qc type. Moreover, these decompositions are invariant under the action of $\Map$ (see section \ref{sec: MCG action}). We have proven

\begin{cor}\label{cor: disjoint union of Teichspaces}
$H(\Sigma)$  as well as  the spaces $\mathcal{C}_{P}$,  $\mathcal{C}$, 
 and $\mathcal{C}^{\prime}$   are partitioned into the disjoint union of  Teichm\"uller subspaces. In addition,  if  a Teichm\"uller subspace is comprised of geodesically complete structures    then there is an injective  continuous  mapping of the  Teichm\"uller space to the Teichm\"uller subspace. In particular, $\mathcal{C}$ is partitioned into the disjoint union of 
 Teichm\"uller spaces- one for each qc-type of Riemann surface structure on 
 $\Sigma$.
\end{cor}

\subsection*{Zig-zag paths  in $\text{H}(\Sigma)$:}
In this subsection, we study the (path) connectivity of various subspaces of $H(\Sigma)$. 

Let  $P$ be a property that is an invariant under quasiconformal mappings,  and $z, w$  two marked hyperbolic structures  in 
$\mathcal{C}_P\subset H(\Sigma)$ having FN-coordinates 
$\vec{z}$ and $\vec{w}$, respectively. By a {\it zig-zag} path we mean a continuous path from $\vec{z}$ to $\vec{w}$ comprised of possibly infinitely many linear (in the FN coordinates) paths where each linear path is in the  $i^{th}$-coordinate plane. More precisely, the first linear path goes from $\vec{z}=(z_1,z_2,z_3,...)$ to 
$(w_1,z_2,z_3,...)$, the second from  $(w_1,z_2,z_3,...)$ to 
$(w_1,w_2,z_3,...)$, and so on. 
Putting these possibly infinitely many paths together, it is not difficult to see that we have constructed   
a  zig-zag  path $L: [0,\infty) \rightarrow  \mathcal{C}_P$ from the marked structure with coordinates $\vec{z}$ to the marked structure with coordinates $\vec{w}$ with the property that the distance  from $L(t)$ to $\vec{w}$ is less than the distance from $L(s)$ to $\vec{w}$ for $t \geq s$, and the image, $L[0,\infty)$, by Lemma \ref{lem: change an FN-coordinate}, is contained in the 
Teichm\"uller subspace $\mathcal{T}(\vec{z})$, where $\mathcal{T}(\vec{z})$ is the unique Teichm\"uller subspace passing through the marked hyperbolic structure $\vec{z}$.
Noting that  the coordinates of  $L(s)$ converge to $\vec{w}$, we have that $\lim_{s \rightarrow \infty} L(s) =\vec{w}$ in the product topology.  Finally noting that $\vec{w} \in \mathcal{C}_P$ using Lemma \ref{lem: change an FN-coordinate}, and the qc invariance of property $P$, we have that  the zig-zag path is in $\mathcal{C}_P$. See schematic Figure \ref{fig:zig-zag} for the 
subspace $\mathcal{C}$. We have proved

%%%%%%%%%FIGURE%%%%%%%%%%%%%%%%
\begin{figure}[htb]
 \begin{center}
     \includegraphics[width=5.5in]{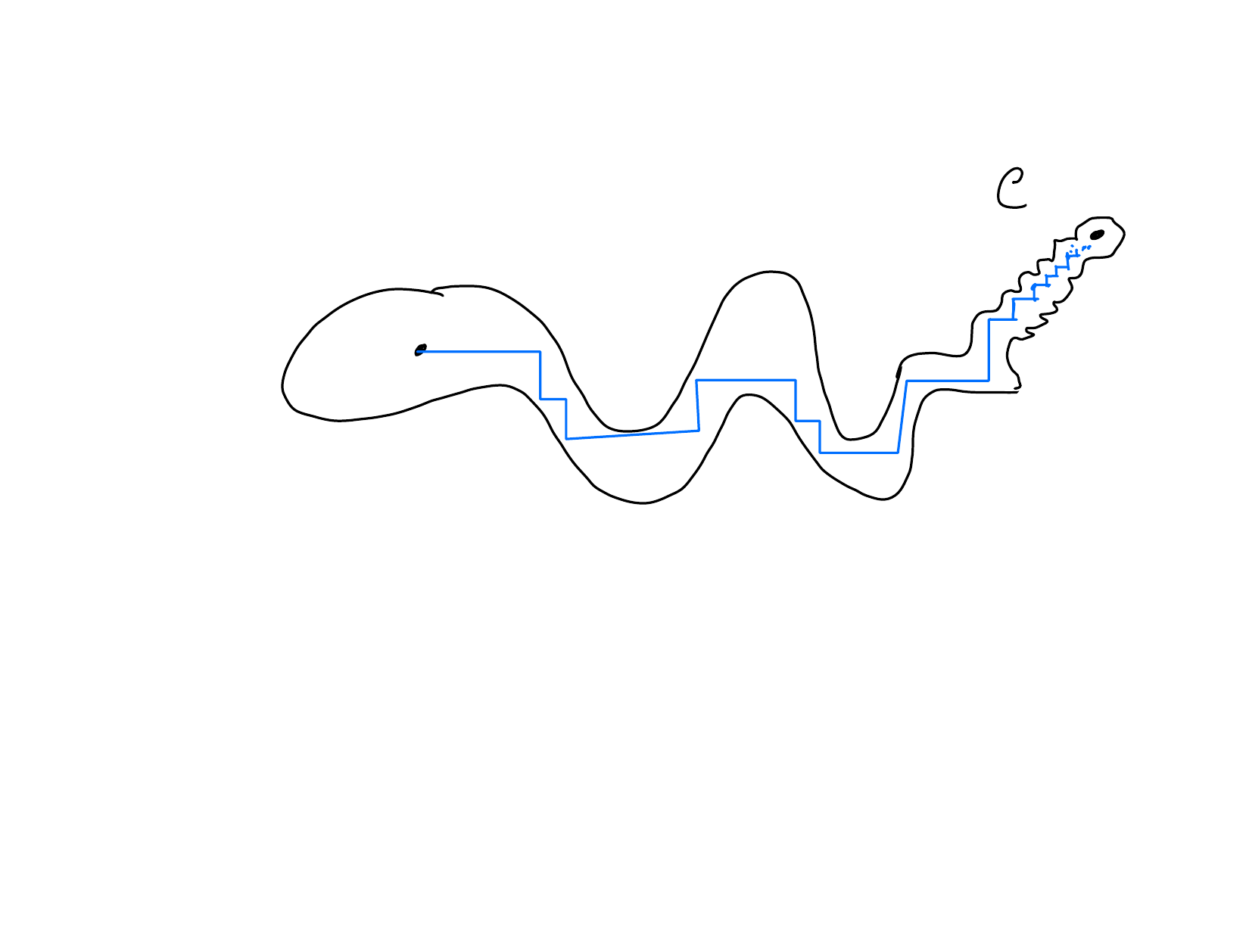}
     \vspace{-100pt}
    \caption{Zig-zag path}
    \label{fig:zig-zag}
  \end{center}
\end{figure}
%%%%%%%%%%%Figure%%%%%%%%%%%%

\begin{lem} \label{lem: qc zig-zag path}
	Given $\vec{z}$ and $\vec{w}$ in $\mathcal{C}_{P} \subset H(\Sigma)$, there exists a zig-zag path 
	$L :[0,1] \rightarrow \mathcal{C}_{P}$ from $\vec{z}$ to $\vec{w}$, where 
	$L(t)$ for any $t \in [0,1)$ is in the same Teichm\"uller subspace as $L(0) = \vec{z}$.
\end{lem}
Equivalently, the Lemma  in terms of hyperbolic structures is

\begin{lem} \label{lem: zig-zag path}
Given   $(X,f)$ and  $(Y, g)$ in  $\mathcal{C}_{P} \subset H(\Sigma)$, there exists a  zig-zag path $L:[0,1] \to  \mathcal{C}_{P}$ from  $(X,f)$ to $(Y,g)$,
where 
	$L(t)$ for any $t \in [0,1)$ is in the same Teichm\"uller subspace as $L(0) =(X,f)$.
\end{lem}

\begin{thm}  \label{thm: property P connectivity} 
 The subspace
   $\mathcal{C}_{P} \subset H(\Sigma)$ is   locally path connected,  connected and the subspace of metrically complete structures $\mathcal{C}^{\prime} \subset H(\Sigma)$ is  path connected. 
 In particular, the subspace of geodesically complete structures 
  $\mathcal{C} \subset H(\Sigma)$ is locally path connected, connected. 
\end{thm}
\begin{proof}
Lemma \ref{lem: qc zig-zag path} combined with the fact that the distance to $\vec{w}$ decreases along $L(s)$ allows us to conclude that
$\mathcal{C}_{P}$ is connected and locally path connected.  Since being 
a geodesically complete structure is a quasiconformal property, we can conclude the same for  $\mathcal{C}$.

To prove that $\mathcal{C}^{\prime}$ is  path connected we first note that 
 If  $X$ is a metrically complete hyperbolic surface then its geodesic completion   either requires attaching  nothing (in this   case that $X$ is geodesically complete)  or a union of funnels.   Here one needs a concatenation of two zig-zag paths. This is because there is an  obstruction to constructing a zig-zag path from $\vec{z}$ to $\vec{w}$. Namely,  an isolated planar end may be a cusp for the structure $\vec{z}$   but a funnel for the structure  $\vec{w}$. Thus to prove that $\mathcal{C}^{\prime}$ is path connected, we need  a concatenation of two zig-zag paths. 
 To this end, using Fenchel-Nielsen coordinates let $\vec{z}, \vec{w} \in \mathcal{C}^{\prime}$, and let $\vec{x} \in \mathcal{C}^{\prime}$ be such that the end geometry of every isolated planar end is a cusp. Construct a path $L_1: [1,\infty) \to \mathcal{C}^{\prime}$ from $\vec{z}$ to $\vec{x}$ as follows. On non-peripheral coordinates, this is just a zig-zag path as above. If $\gamma$ is a peripheral curve, then the corresponding coordinate along this path is set to $z_\gamma/s$ for $s \in [1, \infty]$. Similarly, we construct a path $L_2^{-1}$ from $\vec{w}$ to $\vec{x}$. Reparametrizing and concatenating yields the path $L_2 \ast  L_1$ from $\vec{z}$ to $\vec{w}$.
\end{proof}

\section{The space of complete structures is not convex}
\label{sec: The space of complete structures is not convex}
Let $\Sigma$ be of topologically infinite type with a topological pants 
decomposition.  Recall that  $\mathcal{C}  \subset H(\Sigma)$ ($\mathcal{C}^{\prime}   \subset H(\Sigma)$) is 
the subspace of geodesically complete (resp. metrically complete) 
hyperbolic structures. The Fenchel-Nielsen coordinates 
impose an obvious  straight line  geometry on $H(\Sigma)$, where  the pants decomposition  determines the coordinate system.

 Let  $\mathcal{D}$ be a  subspace of $H(\Sigma)$. We say that 
 $\mathcal{D}$ is {\it convex}  if for any $\vec{z}, \vec{w} \in \mathcal{D}$, the line segment 
 $$(1-s)\vec{z} +s\vec{w} \in \mathcal{D},  \text{  for all  }  0 \leq s \leq 1. $$
It is natural to ask if $\mathcal{C}$ or  $\mathcal{C}^{\prime}$ is convex.

\begin{thm} \label{thm: non-convexity} Suppose $\Sigma$ is of infinite topological type. With respect to any   pants decomposition the subspace of geodesically complete convex hyperbolic structures $\mathcal{C}  \subset H(\Sigma)$ is not convex.
The same  holds for the subspace  $\mathcal{C}^{\prime}$ of metrically complete structures.
\end{thm} 

\begin{lem}\label{lem: embedded flute}
Suppose $\mathcal{P}$  is a topological pants decompostion of $\Sigma$.
Then there exists an embedded   flute subsurface   in $\Sigma$ having a  pants decomposition whose pants curves are in $\mathcal{P}$. 
\end{lem}

\begin{proof}
Consider the  trivalent graph dual to the pants decomposition  
$\mathcal{P}$ with each edge of length one. This graph has a maximal  tree subgraph $T$- this corresponds to a maximal embedded planar domain.  Next choose a vertex $v \in T$, and a geodesic ray (path that minimizes  distance on $T$). This geodesic ray together with the edges emanating from it constitute a flute subsurface with pants decomposition whose curves are inherited from $\mathcal{P}$. 
\end{proof}

\vskip5pt

\begin{rem}
If $\Sigma$ is a tight flute surface (recall that all isolated planar ends are assumed to be cusps)  then we will consider  two (marked) convex  hyperbolic structures on $\Sigma$ which are half-twist flutes.  The straight line between them will not stay in $\mathcal{C}$.
\end{rem}

\noindent{\bf The Strategy:}  Now if $\Sigma$ is a flute surface we do as suggested in the remark above. Otherwise, we start with a pants decomposition of 
$\Sigma$ and an embedded flute surface with an inherited pants decomposition.

We first take a detour to prove an inequality for pairs of pants. 
Define 
$$r(x)=\text{arc}\sinh \left(\frac{1}{\sinh x}\right).$$

\begin{lem}\label{lem: collar lemma}
Consider a hyperbolic pair of pants with  geodesic boundary having lengths $\ell_{1}, \ell_{2}$,  $\ell^{\prime} \geq 0$ and let $d$ be the orthodistance between the geodesics of length $\ell_{1}$ and $\ell_{2}$. Then
$$
r\left(\frac{\ell_{1}}{2}\right)+ r\left(\frac{\ell_{2}}{2}\right)\leq d \leq   
r\left(\frac{\ell_{1}}{2}\right)+ r\left(\frac{\ell_{2}}{2}\right) +
 \frac{\ell^{\prime}}{2}.
$$
\end{lem}

%%%%%%%%%FIGURE%%%%%%%%%%%%%%%%
\begin{figure}[htb]\label{fig: pants}
 \begin{center}
     \includegraphics[width=5.5in]{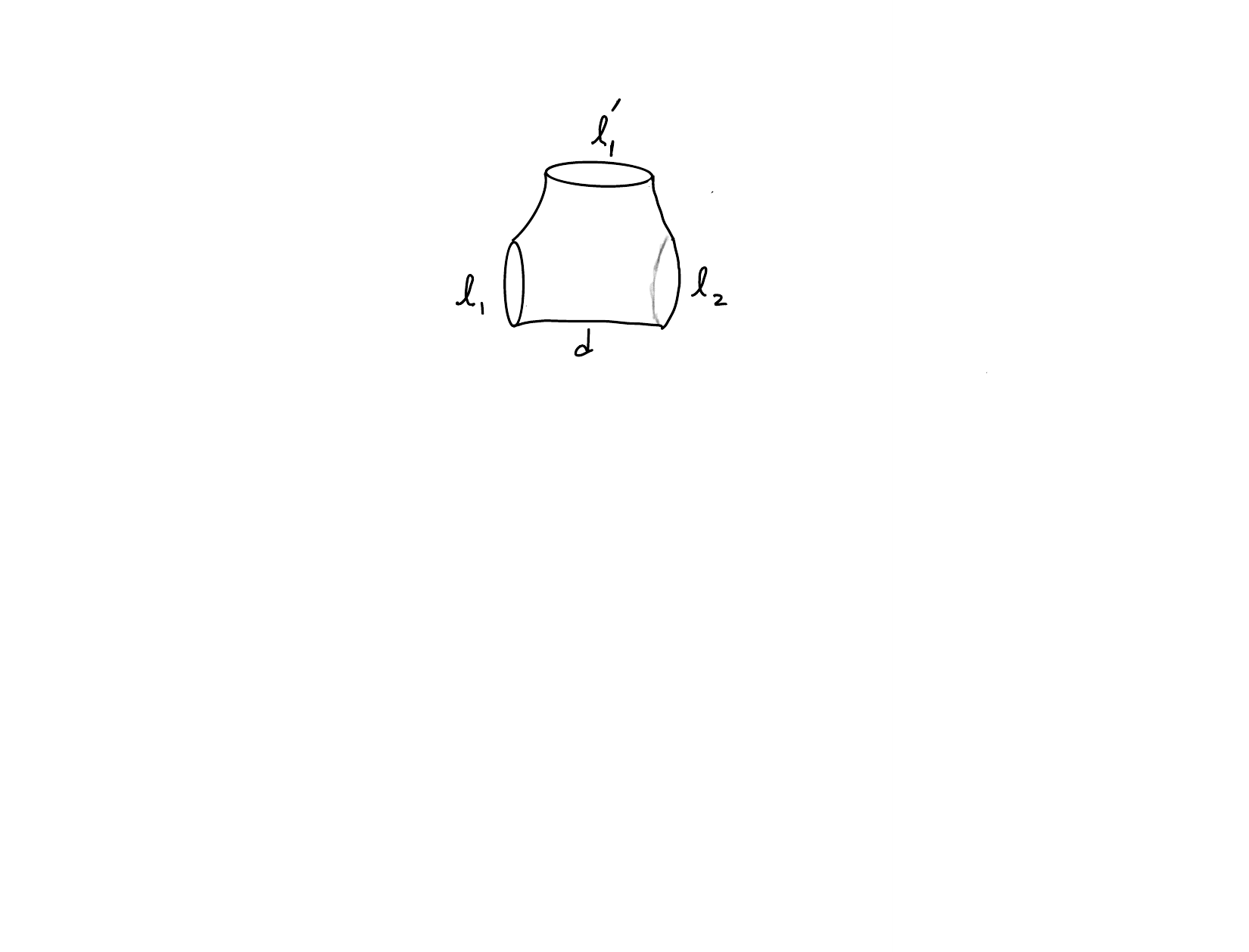}
     \vspace{-175pt}
    \caption{$a=\frac{\ell_{1}}{2},b=\frac{\ell_{2}}{2}, \text{ and } c= \frac{\ell_{1}^{\prime}}{2}$}
    %\label{fig:setting}
  \end{center}
\end{figure}
%%%%%%%%%%%Figure%%%%%%%%%%%%

\begin{proof} 
The lower bound  follows from the collar lemma. We outline the proof of the upper bound. 

The pair of pants is the union of 
two isometric copies of a right-angled hexagon glued along three seams. This hexagon has alternating side lengths $a,b,$ and $c$, where 
$a=\frac{\ell_1}{2}, b= \frac{\ell_2}{2},$ and $c= \frac{\ell^{\prime}}{2}$.
The side joining $a$ and $b$ is denoted  $d$.  The 
hexagon  decomposes into two pentagons with the side $c$ divided into two segments $c_{1}$ and $c_{2}$. Using the pentagon formula  for hyperbolic geometry, a computation yields
\begin{displaymath}
d=\text{arc}\sinh \left(\frac{\cosh c_{1}}{\sinh a}\right)+\text{arc}\sinh 
\left(\frac{\cosh c_{2}}{\sinh b}\right)
\end{displaymath}
\begin{equation}\label{eq: arc sinh inequality}
\leq  \text{arc}\sinh \left(\frac{1}{\sinh a}\right)+\text{arc}\sinh 
\left(\frac{1}{\sinh b}\right)
+\log (\cosh c_{1}  \cosh c_{2})
\end{equation}
\begin{equation}\label{eq: summation}
\leq \text{arc}\sinh \left(\frac{1}{\sinh a}\right)+\text{arc}\sinh 
\left(\frac{1}{\sinh b}\right)
+\log (\cosh c)
\end{equation}
\begin{equation}
\leq \text{arc}\sinh \left(\frac{1}{\sinh a}\right)+\text{arc}\sinh 
\left(\frac{1}{\sinh b}\right)
+c
\end{equation}
Where inequality (\ref{eq: arc sinh inequality}) follows from the fact that
$\text{arc}\sinh x < \log x$, and inequality (\ref{eq: summation}) from the summation formula for $\cosh (c_1+c_2)$.
Finally,  replacing $a,b,$ and $c$ in terms of the  geodesic lengths of the pants curves yields the desired upper bound. 
\end{proof}

\begin{proof}[Proof of Theorem  \ref{thm: non-convexity}]
We prove the theorem for  $\mathcal{C}$. The argument for $\mathcal{C}^{\prime}$
is the same. 
Let $\Sigma$ be an infinite type surface with a topological pants decomposition.
Recall that  any isolated planar end is given a   cusp  geometry by decree.  Next  let $\sigma$ be a topologically embedded flute with boundary with an induced pants decomposition as guaranteed by  Lemma \ref{lem: embedded flute}. Call these pants curves $\{\gamma_{n}\}$, and the peripheral curves 
$\{\gamma^{\prime}_{n}\}$.  Of course,  if any of the isolated planar ends of 
$\sigma$ are also isolated planar ends  in 
$\Sigma$ then we make them cusps and designate the length of 
$\gamma^{\prime}_{n}$ to be zero. Otherwise,  we designate  the length of 
$\gamma^{\prime}_{n}$  to be $\ell_{n}^{\prime}$,  where 
$\sum \ell_{n}^{\prime} <\infty.$  These are the pants curves of $\sigma$ which are not peripheral on $\Sigma$.  We next  put a hyperbolic structure on this  flute subsurface $\sigma$. 

Let $Y_0$ and $Y_1$ be    hyperbolic structures on $\sigma$  with 
non-peripheral pants curves 
$\{\gamma_{n}\}$ having  coordinates 
$\{(4 \log n, \frac{1}{2})\}$ and $\{(4 \log n, -\frac{1}{2})\}$, 
resp., with the length of the $n^{th}$  peripheral curve being $\ell_{n}^{\prime}$
as above. These subsurfaces are called half-twist flutes. 
 The underlying hyperbolic structures of $Y_{0}$ and $Y_{1}$
are isometric, just their markings are different. It is a 
consequence of Theorem 9.7  in \cite{BHS}  that the non-isolated end of $Y_i$, $i=0,1$ is   complete.  
  Next we construct two geodesically complete hyperbolic structures $X_0$ and $X_1$. 
  
 Recall that we have a fixed pants decomposition on $\Sigma$, an embedded 
 flute surface $\sigma$, and an induced pants decomposition on $\sigma$.
  With the geometry constructed above  for $Y_1$ and $Y_2$, we isometrically embed $Y_{i}$ in $X_{i}$, for $i=0,1$.   For the  complement of the flute 
  $\Sigma -\sigma$ set the length of all the pants curves to be 1. The twist parameters in   $\Sigma -\sigma$ don't play a role and  can be arbitrarily chosen. We have now constructed our two hyperbolic structures 
  $X_0$ and $X_{1}$.
  
  We first claim that    $X_0$ and $X_{1}$ are in $\mathcal{C}$. That is, all geodesic rays are complete (do not escape in finite time). This is because 
  any geodesic ray exiting $X_i$ ($i=0,1$) either
  \begin{enumerate}
  \item the ray is eventually disjoint from $Y_i$, and thus passes through infinitely many simple closed geodesics of length 1 
  \item the ray eventually stays in the flute $Y_{i}$, and thus by construction does not escape  the non-isolated end of $Y_{i}$,  or
  \item enters and exits $Y_{i}$ infinitely often.  But then the geodesic ray passes through infinity many of the $\{\gamma_{n}^{\prime}\}$, and since their collar widths go to infinity, the ray is again complete. 
  \end{enumerate}
  
  Next consider the straight line segment in $H(\Sigma)$  from $X_{0}$ to $X_{1}$. We only focus on the flute coordinates since the other coordinates stay constant  along the line segment.  Let 
  $\vec{z_0}=(z_{0}^{1},...,z_{0}^{n},...)$ and   $\vec{z_1}=(z_{1}^{1},...,z_{1}^{n},...)$ be the Fenchel-Nielsen coordinates for flute subsurfaces $Y_0$ and $Y_1$, resp.
  Here the $n^{th}$-coordinate of $Y_0$ is 
  $z_{0}^{n}=4\log n+\frac{i}{2}$, and the $n^{th}$-coordinate of  $Y_1$
  is $z_{1}^{n}=4\log n-\frac{i}{2}$.
 For $0\leq s \leq 1$, the $n^{th}$-coordinate of  the straight line 
 $(1-s)\vec{z_0}+s \vec{z_1}$ is 
 $4\log n+\frac{i}{2}(1-2s)$. In particular, when $s=\frac{1}{2}$, the 
 coordinate is $4 \log n$. This hyperbolic structure (which we denote $X_{\frac{1}{2}}$)  is a 
 $0$-twist flute and hence, using Lemma \ref{lem: collar lemma},  the distance from $\gamma_{1}$ to   $\gamma_{N}$ equals
 $$
 \sum d_{n} \leq \sum_{n=1}^{N} \left(r \left( 2\log n \right) 
 +r \left(2\log n+1 \right)+
 \frac{\ell_{n}^{\prime}}{2}\right)
 $$
 But $\sum \ell_{n}^{\prime} < \infty$ by construction, and since
 $r(2\log n) +r(2\log n+1)$ is asymptotic to $2(\frac{1}{n^2} +\frac{1}{(n+1)^2})$
 as $n$ goes to $\infty$, we see that all the terms in the right hand side above
 converge. Hence $X_{\frac{1}{2}}$ is not complete, and thus 
 $\mathcal{C}$ is not convex. 
\end{proof}

\section{Mapping class group action on $H(\Sigma)$} \label{sec: MCG action}

Recall that $\Map\Sigma)$ acts on $H(\Sigma)$, and 
observe that it  keeps invariant the subspace  
$\mathcal{C}_{P}\subseteq H(\Sigma)$. In fact, we have 

\begin{lem}
	$\Map(\Sigma)$ acts on $H(\Sigma)$  as a group of homeomorphisms.
\end{lem}

\begin{proof}
	Any $\phi \in \Map (\Sigma)$ induces a bijection from $H(\Sigma)$ to itself defined by $(X,f) \mapsto (X, f \circ \phi^{-1})$. Consider sets of the form $V = (V_1, V_2, \cdots ) \subset H(\Sigma)$ where each $V_i$ is an open subset of the upper half-plane, $\text{U}$, and $V_i = \text{U}$ for all but finitely many $i \in \NN$. These form a basis for the topology on $H(\Sigma)$. 
Each such set $V$ gets mapped to an open set of the same form. The same argument applied to  
$\phi^{-1}$ allows us to conclude  that $\phi$ acts as a homeomorphism on 
$H(\Sigma)$.
\end{proof}

\subsection*{Mapping classes and  their action on the Teichm\"uller subspaces:}

Suppose $g \in \Map$. We say that $g$  is a {\it qc mapping class with respect to the hyperbolic structure}  $(X, h)$ if there exists a qc-mapping
$q:X \rightarrow X$  so that the  diagram below commutes up to isotopy,

\[ \begin{tikzcd} \label{fig: commutative diagram}
	\Sigma \arrow{r}{g} \arrow[swap]{d}{h} & \Sigma \arrow{d}{h} \\%
	X \arrow{r}{q}& X
\end{tikzcd}
\]

We have the following as a consequence of Lemma 
\ref{lem: Teich subspaces disjoint},

\begin{prop} Let $g \in MCG(\Sigma)$. The following are equivalent.
	\begin{enumerate}
		\item $g$  is a  qc mapping class with respect to $(X,h)$,
		\item  $g$ keeps invariant $\mathcal{T} (X,h)$
	\end{enumerate}
\end{prop}

We have seen that $H(\Sigma)$ is the disjoint union of Teichm\"uller subspaces, and the mapping class group permutes the Teichm\"uller subspaces. We have a similar statement for the complete structures in $H(\Sigma)$. 

Not every surface $\Sigma$ admits a hyperbolic structure whose mapping class group orbit in $H(\Sigma)$ is dense. A \textit{nondisplaceable subsurface} is a compact subsurface $S \subset \Sigma$ such that $\phi(S) \cap S \neq \emptyset$ for all $\phi \in \Map(\Sigma)$.

\begin{lem} Let $\Sigma$ be an infinite type surface.
	\begin{enumerate}
		\item  Let $g \in \Map(\Sigma)$. The element $g$ either keeps invariant a Teichm\"uller subspace or moves it away from itself.	
		\item For any quasiconformally invariant property $P$, the mapping class group acts invariantly on the subspace $\mathcal{C}_P \subset H(\Sigma)$.
		\item  $\Sigma$ carries a convex hyperbolic structure with dense orbit if and only if there are no nondisplaceable subsurfaces.
	\end{enumerate}

\end{lem}
\begin{proof}
	The proofs of items (1) and (2) are immediate. 
	
	Now to prove item (3), suppose $\Sigma$ has no nondisplaceable subsurfaces. Let $\{\Sigma_i\}$ be a compact exhaustion. Each $\Sigma_i$ has countable many translates $\{\Sigma^j_i\}$ that are pairwise disjoint and $\Sigma_i^j \cap \Sigma_k^l = \emptyset$ if $(i,j) \neq (k,l)$. We now build a hyperbolic structure on $\Sigma$ whose mapping class group orbit is dense in $H(\Sigma)$. For each $(i,j)$ there is a mapping class $\phi _i^j$ such that $\phi_i^j(\Sigma_i^j) = \Sigma_i$. We put a hyperbolic structure on each $\Sigma_i^j$ such that their push forward under the $\phi_i^j$ maps are dense in $\mathcal{T}(\Sigma_i)$. Extend this arbitrarily to a marked hyperbolic structure on $\Sigma$ say, $(X,f)$. Now let $(Y, g) \in H(\Sigma)$, for each $i \in I$ there is a sequence of maps $\phi_i^j$ such that $f((\phi_i^j)^{-1} (\Sigma_i))$ converges to $g(\Sigma_i)$. Since $(Y,g)$ is arbitrary, we conclude that $(X,f)$ has a dense mapping class group orbit. 
	
	For the other direction, assume $\Sigma$ has a nondisplaceable subsurface $S$. We can pick a pants decomposition $\Gamma$ of $\Sigma$ that restricts to a pants decomposition of S. If there is a hyperbolic structure $f:\Sigma \to X$ that has a dense orbit, then for any pants curve $\gamma$ in $S$, we have that the lengths in $X$ of the mapping class group orbit of $\gamma$ are  dense in $\mathbb{R}_{>0}$. Since $S$ is non-displaceable, every curve in the orbit of $\gamma$ intersects $S$. Since $S$ is compact, the lengths (with respect to $X$) of simple closed curves intersecting $S$ is discrete (see Lemma 3.1 in \cite{BaKim}). This is a contradiction. 
\end{proof}

\section{The trichotomy} \label{sec: trichotomy}

In the finite type setting, all mapping classes admit quasiconformal representatives. In the infinite type setting, there is a distinguished collection of {\it never quasiconformal mapping classes};  that is, they admit no quasiconformal representatives for any hyperbolic structure on the surface. 

Let  $\mathcal{D} \subset H(\Sigma)$ be a subspace which is path connected and invariant under the mapping class group.  Then for $\phi \in \Map(\Sigma)$ we say that $\phi$ relative to $\mathcal{D}$ is
	\begin{itemize}
	\item  {\it always quasiconformal} if for any hyperbolic structure $(X,f) \in \mathcal{D}$, $\phi$ admits a quasiconformal representative.
	\item {\it sometimes quasiconformal} if there exists hyperbolic structures $(X,f), (Y,g) \in \mathcal{D}$ such that $\phi$ admits a quasiconformal representative with respect to $(X,f)$ but not $(Y,g)$.
	\item {\it never quasiconformal} if there does not exist a hyperbolic structure $(X,f) \in \mathcal{D}$ for which $\phi$ admits a quasiconformal representative. 
	\end{itemize}

If there is no mention of  relative to a subspace then it is understood that the subspace is the full space $H(\Sigma)$.

Note that the three classes above are mutually exclusive and each $\phi \in \Map(\Sigma)$ falls into exactly one of the classes and therefore defines a trichotomy. In fact, this trichotomy is invariant under conjugation and therefore descends to any quotient of the mapping class group. In this section, we discuss this trichotomy from several perspectives. 
\vskip5pt

\noindent{\bf Analytic:}
The definition is given in terms of the existence of quasiconformal representatives which is an analytic notion. Equivalently, we can instead consider the existence of a hyperbolic structure $X$ such that $\phi$ acts with finite translation length on $\mathcal{T}(X)$. This view point is more geometric and directly analogous to Bers' proof in \cite{Bers}.
\vskip5pt 

\noindent{\bf Dynamical:}
The mapping class group also acts as permutations of the Teichm\"uller subspaces in $H(\Sigma)$. From this perspective, the trichotomy can be stated in terms of the existences of fixed points (ie invariant Teichm\"uller subspaces). In fact this perspective extends to a trichotomy for mapping classes acting on many objects associated to surfaces and in many cases the breakdown is the same. For example: $\mathcal{C}_P$, Pants graph, marking graph, and the flip graph.
\vskip5pt

\noindent{\bf Topological:}
For finite type surfaces, the Nielsen-Thurston classification, can be stated (and proved) in terms of the action of the mapping classes on simple closed curves. In this subsection, we give a topological characterization of always quasiconformal mapping classes.    
  
\vskip5pt
\noindent{\bf Definition.} 
	A mapping class $\phi \in \Map(\Sigma)$ is said to be {\it finitely supported} if outside some finite type subsurface, $\phi$ is isotopic to the identity. 
\vskip5pt
\noindent{\bf Definition.}
	We say that a mapping class, $\phi$ {\it acts non-trivially} on a simple closed curve $\alpha$ if either of the following two conditions hold: 
	\begin{enumerate}
		\item $\phi(\alpha) \neq \alpha$ 
		\item $\phi$ does not acts as the identity on an annular neighborhood of $\alpha$.
	\end{enumerate}

\begin{thm}\label{thm: finite support}
	A mapping class $\phi$ is always quasiconformal if and only if $\phi$ is finitely supported. 
\end{thm}

\begin{proof}
	If $\phi$ has finite support then it is supported on a finite type (possibly disconnected) subsurface $S\subset \Sigma$. Since $S$ is finite type, there is a quasiconformal map between any two complex structures on $S$. Therefore $\phi$ is always quasiconformal since it has a representative that acts as a quasiconformal map on $S$ and as the identity on $\Sigma \setminus S$.
	
	We will prove the other direction by showing that if $\phi$ is not finitely supported, then there exists a hyperbolic structure $X$ whose Teichm\"uller subspace is not fixed by $\phi$. Assume $\phi$ is not finitely supported. Then there exists a disjoint sequence of simple closed curves, $\{\gamma_i\}$, leaving every compact set such that $\phi$ acts non-trivially on each curve.
	
	Since $\phi$ acts non-trivially on each $\gamma_i$, by passing to an infinite subsequence we can assume that at least one of the following is true:
	\begin{enumerate}
		\item $i(\phi(\gamma_i) , \gamma_i) > 0$ for all $i \in \ZZ$ 
		\item $i(\phi(\gamma_i) , \gamma_i) = 0$ and $\phi$ does not acts as the identity on an annular neighborhood of $\gamma_i$ for all $i \in \ZZ$.
	\end{enumerate} 
	
	If the first condition holds, by passing to a subsequence, we can further assume that $i(\phi(\gamma_i) , \gamma_j ) =0$ for all $i\neq j$. We can pick a hyperbolic structure $(X,f)$ on $\Sigma$ where lengths of $\{\gamma_i\}$ go to $0$,  as $i$ goes to $\infty$. Since the $\phi(\gamma_i)$ intersects $\gamma_i$, the length of $\phi(\gamma_i)$ is bounded below by the width of the collar around $\gamma_i$. By the collar lemma, the ratio of lengths 
	$\frac{\ell_{X}(\phi(\gamma_i))}{\ell_{X}(\gamma_i)}$ goes to infinity as 
	$\ell_{X}(\gamma_i) \to 0$ and Wolpert's lemma tells us that $\Phi$ does not have a quasiconformal representative fixing $\mathcal{T}(X,f)$.
	
	Now assume that instead, the second condition holds. We can again pass to an infinite subsequence so that $i(\phi(\gamma_i) , \gamma_j ) =0$ for all $i\neq j$. For each $\gamma_i$, either $\phi(\gamma_i) = \gamma_i$ or $\phi(\gamma_i) \neq \gamma_i$. 
	
	If infinitely many $\gamma_i$ are fixed by $\phi$, we specify with Frenchel-Nielsen coordinates a hyperbolic structure, $X$, by expanding $\{\gamma_i\}$ to a pants decomposition $\Gamma$ and choosing length parameters that go to $\infty$ while all twists parameters are trivial. In $X$, let $\alpha_i$ be the shortest curve intersecting $\gamma_i$. The lengths of $\alpha_i$ go to $0$ as the length of $\gamma_i$ goes to $\infty$. Since $\phi(\gamma_i) = \gamma_i$, then $\phi$ restricted to an annular neighborhood of $\gamma_i$ is a Dehn twist. The triangle inequality then tells us that the length of $\phi(\alpha_i)$ is bounded below by the length of $\gamma_i$. Since the ratio of the lengths $\frac{\ell_{X}(\phi(\alpha_i))}{\ell_{X}(\alpha_i)}$ can be made arbitrarily large, $\phi$ can not fix such a $\mathcal{T}(X,f)$.
	
	If instead, infinitely many $\gamma_i$ are not fixed by $\phi$, we can extend the set $\{\gamma_i\} \cup \phi(\{\gamma_i\})$ to a pants decomposition. Again, we specify a hyperbolic structure, $X$ by choosing lengths of $\gamma_i$ and $\phi(\gamma_i)$ such that the ratio of the lengths $\phi(\alpha_i)/\alpha_i$ become arbitrarily large, $\phi$ can not fix such a $\Teich(X)$.
\end{proof}

Since any compactly supported mapping class is a finite composition of Dehn twists and there are only countably many simple closed curves on $\Sigma$, we get the following corollary.

\begin{cor}
	The set of always quasiconformal mapping classes is countable. 
\end{cor}

 \subsection{Sometimes quasiconformal maps.} This class contains all isometries, shifts, infinite multi-twists, and more. This class is closed under taking powers but not closed under composition.
 
 We next  use a  result of  Matsuzaki to show that the set of sometimes quasiconformal mapping classes is uncountable.  Let $\Sigma$ be an infinite type surface and $\{\gamma_{i}\}$ a disjoint set of 
 homotopically distinct  oriented simple closed curves.  Let $D_i$ be Dehn twist about $\gamma_{i}$.
 
 \begin{prop}
 For any sequence    $\{n_i \in \ZZ\}$ with infinitely many $n_i \neq 0$,
  the infinite multi-twist $ \prod_{i \in I} D_i ^{n_i}$ is sometimes quasiconformal. 
 \end{prop}
 
 \begin{proof}   The proof is an application of Theorem 1 in   \cite{Ma1}. Namely,
 given the sequence  $\{n_i\}$ choose the hyperbolic structure $X$  so  that  the lengths  of the $\gamma_{i}$ go to infinity.  Now the lower bound in 
 Theorem 1 of \cite{Ma1} implies that there is not qc representative. On the other hand by choosing the lengths of the $\gamma_{i}$ to go to zero faster than the $n_{i}$ go to infinity then the right-hand inequality of   Theorem 1 of \cite{Ma1}   shows that the dilatation is bounded above and hence $\phi$ does have a qc representative. Thus we conclude that the mapping class is sometimes qc.
  \end{proof}

 As a consequence there are uncountably many sometimes qc mappings.

  \begin{rem}
  In the proof above, using  \cite{BaSa},  each   choice of $X$ can be made so that  $X$ is either  in $\mathcal{C}$ or $\mathcal{C}^{\prime}$. This is accomplished by choosing appropriate twists. 
  \end{rem}

 In the next section we construct  uncountably many never quasiconformal mapping classes.  
  
 \begin{prop}
 The class types   $\{\text{Always qc classes}\},\{\text{Sometimes qc  classes}\},$ and $\{\text{Never qc  classes}\}$ in $\Map(\Sigma)$ are invariant  under conjugation. Moreover, since always qc mapping classes form a group, this group is normal in $\Map(\Sigma)$.
 \end{prop}
 
 We have proven

\begin{thm}  \label{thm: type classification}
 $\Map(\Sigma)$ acts faithfully as a group of homeomorphisms of 
$H(\Sigma)$,  keeps invariant $\mathcal{C}_{P} \subset H(\Sigma)$, and 
permutes  the set of Teichm\"uller subspaces of $\mathcal{C}_{P}$. 
Moreover, for $\phi \in \Map(\Sigma)$ 
\begin{itemize}
\item $\phi$ is always quasiconformal rel. $\mathcal{C}_{P}$  if and only if  $\phi$ keeps invariant  every Teichm\"uller subspace in $\mathcal{C}_{P}$  if and only if $\phi$ is finitely  supported. 

\item $\phi$ is sometimes  quasiconformal rel. $\mathcal{C}_{P}$ if and only if there exists at least one Teichm\"uller subspace  in $\mathcal{C}_{P}$  that is kept invariant and at least one in $\mathcal{C}_{P}$ that is not kept invariant. 

\item $\phi$ is never quasiconformal rel. $\mathcal{C}_{P}$ if and only if   there is no Teichm\"uller subspace in  $\mathcal{C}_{P}$ kept invariant by 
 $\phi$.
\end{itemize}
\end{thm}

\section{Constructing never qc maps} \label{sec: Constructing never qc maps}
In this section we apply a theorem of Bers (Theorem 2 in \cite{Bers}) to construct an uncountable class of never quasiconformal mapping classes on any infinite type surface $\Sigma$. This class is constructed by taking high powers of partial pseudo-Anosov maps on disjoint finite subsurfaces. Similar constructions of such mapping classes have been used in other studies such as \cite{DFH}.

Let $\{\Sigma_i\}_{i \in I}$ be an infinite collection of pairwise disjoint finite type subsurfaces of $\Sigma$. Let $\phi_i \in \Map{(\Sigma_i)} \subset \Map{(\Sigma)}$ be a pseudo-Anosov mapping class such that $\inf{\tau(\phi_i)} > 0$ where $\tau(\phi_i)$ is translation length of $\phi_i$ in $\mathcal{T}(\Sigma_i)$ with the Teichm\"uller metric. 

\begin{thm}\label{thm: never qc}
	The infinite direct product $\phi = \prod_{i \in I} \phi_i ^{n_i}$ is never quasiconformal for any unbounded sequence $\{n_i \in \ZZ\}$.
\end{thm}

\begin{proof}
	By Bers' theorem, $\tau(\phi_i)$ is realized by a conformal structure of the first kind. Since each $\phi_i$ is a pseudo-Anosov, $\tau(\phi_i^n) = n\tau(\phi_i)$ for any $n \in \ZZ$. Let $X$ be a conformal structure on $\Sigma$ and $X_i$ the induced conformal structure on $\Sigma_i$. The quasiconformal constant of $\phi$ with respect to $X$ is bounded below by the quasiconformal constant of $\phi$ restricted to any subsurface, in particular the dilatation of $\phi_i^{n_i}$ is less than or equal to the dilatation of $\phi$.
	
	Since we chose $\phi$ such that $\{\tau(\phi_i)\}$ is bounded away from $0$, the dilatation of $\phi_i^{n_i}$ goes to $\infty$ as $i$ goes to $\infty$ and we conclude that $\phi$ is not quasiconformal with respect to any conformal structure $X$.
\end{proof}

We remark that the never qc mapping classes of Theorem \ref{thm: never qc}   are never qc relative to any  subspace that is path connected and $\Map$ invariant. In particular, they are never qc relative to the geodesically complete structures $\mathcal{C}$. In section  \ref{sec: rel. trichotomies}, we construct mapping classes that are never qc relative to certain subspaces but not others. 

%%%%%%%%%%%%%%%%%%%%%%%%%%%%

\section{\text{MCG}  action on Teichm\"uller spaces}\label{sec: MCG action 
on Teichmuller spaces}
First we discuss some basics.

\begin{definition}
Let $G_1$ and $G_2$ be actions by homeomorphisms on the topological space $M$. We say that $G_1$ is  orbit eqiuivalent to $G_{2}$  if
for every $p \in M$, the $G_1$ orbit of $p$ equals the $G_2$ orbit of 
$p$.
\end{definition}

Let $X$ be a Riemann surface, $\mathcal{T}(X)$  its associated  Teichm\"uller space, and  $\Map{(X)}$ the  mapping class group of the underlying topological surface $X$. 

\begin{definition}
The {\it modular group} of $X$, denoted $\text{Mod}(X)$, is the  subgroup of 
$\Map(X)$ of isotopy classes which contain a quasiconformal homeomorphism. The automorphism  group of $X$,  denoted $\text{Aut}(X)$,
is the group of conformal self-maps of $X$. 
\end{definition}

Since a conformal self-map  isotopic to the identity is the identity, there is a natural injection  of  $\text{Aut}(X)$ in $\text{Mod}(X)$. So we have the containments,
$$
 \text{Aut}(X) \subset  \text{Mod}(X) \subset \Map(X).
 $$
 
 Now, suppose $G$ is an abstract group, and 
 $\rho : G \rightarrow \Map(X)$  is a representation that  induces an action, denoted $\ast$,  of $G$ on $\mathcal{T}(X)$. Namely, 
 $$
 G \times \mathcal{T}(X) \rightarrow \mathcal{T}(X) \text{ given by}
 $$
 $$
 g \ast [q: X \rightarrow Y] \mapsto [q \circ \rho (g^{-1}) : X \rightarrow Y] 
 $$
 
 We emphasize  that $G$  can be any abstract group. We have
 
 \begin{lem}\label{lem: group action on Teich}
 Suppose the $G$-action induced by $\rho$ on $\mathcal{T}(X)$ is orbit equivalent to the action of $\text{Mod}(X)$.  Then 
 $\rho (G) =  \text{Mod}(X).$
 \end{lem}
 
 \begin{proof} 
 Suppose $g \in G$. Then there exists $q : X \rightarrow X \in 
 \text{Mod}(X)$ so that 
 $$
 [\rho(g^{-1}): X \rightarrow X]:= g \ast [id: X \rightarrow X] = [id \circ q^{-1}]
  $$
  where the last equality is  by  orbit equivalence of the  two actions.  
   But this implies there exists a conformal self-map
 $c: X \rightarrow  X$ so that $\rho(g)$ is isotopic to 
 $q^{-1}  \circ c^{-1}  \in \text{Mod}(X)$. Therefore, 
 $\rho(G) \leq \text{Mod}(X)$. 
 
 To prove the other containment,   
 choose a point $[q : X \rightarrow Y] \in \mathcal{T}(X)$ for which 
   the stabilizer of $[q : X \rightarrow Y]$ in  $\text{Mod}(X)$  is trivial, equivalently $\text{Aut}(Y)$ is trivial.  
    Now, for  $\phi \in \text{Mod(X)}$,   there
 exists $g \in G$ so that 
  $$
 [q \circ \rho(g^{-1}): X \rightarrow Y]:= g \ast [q: X \rightarrow Y] =
  [q \circ \phi^{-1}:X \rightarrow Y]
  $$
  where the last equality is by the orbit equivalence of the $G$ and 
 $\text{Mod}(X)$-actions. Hence, since $\text{Aut}(Y)$ is trivial we have the diagram, 
 
 \begin{equation}\label{diagram: image of Teich space}
	\begin{tikzcd} 
	X   \arrow{r}{q \circ \phi^{-1}} \arrow[swap]{dr}{q \circ \rho (g^{-1}) } & Y  \arrow{d}{id} \\
		& Y
	\end{tikzcd}
\end{equation}
Equivalently, $\rho (g)=\phi$, and thus  $\rho(G) \geq \text{Mod}(X)$ finishing the proof. 
  \end{proof}
  
\noindent{\bf Algebraic Characterization of Dehn Twists:}
  In \cite{BRD} the authors prove that a mapping class has compact support if and only if its conjugacy class is countable. They then combine this with Ivanov's algebraic description of Dehn twists in terms of their centralizers for finite type mapping class groups \cite{Iv} to get an algebraic characterization of Dehn twists in big mapping class groups briefly summarized below.
 
Let $G \leq \Map(\Sigma)$ be a subgroup containing all finitely supported mapping classes. We denote by $\mathcal{F}_G$ the set of elements in $G$ with countable conjugacy class.
We say that an element $f \in G$ is in $\mathcal{M}_G$ if:
\begin{enumerate}
\item $f$ has a countable conjugacy class i.e. $f \in \mathcal{F}_G$,
\item  $Z(\mathcal{F}_G \cap C_G(f))$ is infinite cyclic and generated by $f$ and,
\item $C_G(f) = C_G(f^k)$ for all $k \geq 1$.
\end{enumerate}
 
For $f \in \mathcal{M}_G$, define $(\mathcal{M}_G)_f := \{g \in \mathcal{M}_G \vert fg = gf\}$.
 
We say that $f \in G$ satisfies the \textit{algebraic twist condition} if $f \in \mathcal{M}_G$ and $(\mathcal{M}_G)_f$ is maximal with respect to inclusion. It is shown in \cite{BRD} that if $G = \Map(\Sigma)$ then $f \in G$ satisfies the algebraic twist condition if and only if $f$ is a Dehn twist. In general, if $G$ contains all finitely supported mapping classes then Dehn twists in $G$ are a subset of the elements satisfying the algebraic twist condition.

  \begin{lem}\label{lem: no isomorphism} 
  Let $\Sigma$ be an infinite type surface and $X$ a hyperbolic surface. Then there does not exist a group  isomorphism $\rho: MCG(\Sigma) \to Mod(X)$.
  \end{lem}
  
  \begin{proof}
  Suppose for the sake of contradiction there exists an isomorphism. If $X$ is finite type, then we are done since $MCG(\Sigma)$ is uncountable and $Mod(X)$ is countable. Now we proceed with the case that $X$ is infinite type.

 Now, by the remark above any Dehn twist $T_\alpha \in \Mod(X)$ also satisfies the algebraic twist property. In particular, we see that the pre-image under $\rho$ of a Dehn twist in $\Mod(X)$ is a Dehn twist in $\Map(\Sigma)$. Note that the image of a Dehn twist is not necessarily a Dehn twist. Two Dehn twists commute if and only if they have disjoint or coinciding support. Let $(\alpha_i, \beta_i)$ be a collection of pairs of curves on $X$ such that each pair fills a finite type subsurface $S_i$ with $\lvert \chi (S_i) \rvert > 3$ and $S_i \cap S_j = \emptyset$ for $i \neq j$. We use each pair to build a partial pseudo-Anosov $\phi_i \in \langle T_{\alpha_i}, T_{\beta_i}\rangle$ 
(using e.g. Thurston's construction) such that their composition is never qc as in Theorem \ref{thm: never qc}.
We now check that the preimages $\{\rho^{-1}(\phi_i)\}$ have pairwise disjoint support. This is a consequence of the following: the preimage of any Dehn twist under $\rho$ is a Dehn twist in $\Map(\Sigma)$, two Dehn twists commute if and only if they have disjoint or coinciding support, and our choice of disjoint pairs of intersecting curves $(\alpha_i, \beta_i)$.
Now the infinite composition, say $g$, of the maps $\rho^{-1}(\phi_i)$ is well defined in $\Map(\Sigma)$ and $\rho(g)$ is a never qc map. This contradicts $\Mod(X)$ being the image of $\rho$.
  \end{proof}

  \begin{thm}\label{thm: MCG of finite type} 
  Let $\Sigma$ be any topological surface.  If $\Map(\Sigma)$ acts faithfully on the  Teichm\"uller space $\mathcal{T}(X)$  and is  orbit equivalent to 
  $\text{Mod}(X)$ for some Riemann surface $X$, then $\Sigma$ is of finite type. 
  \end{thm}
  
  \begin{proof}
  Set $G=\Map(\Sigma)$, and let $\rho : G \rightarrow \Map(\Sigma)$  be a faithful representation. If $\rho(G)$ acts on $\mathcal{T}(X)$ then by Lemma  \ref{lem: group action on Teich},
   $\rho(G)=\text{Mod}(X)$ and hence $\Map(\Sigma)$ is isomorphic to 
   $\text{Mod}(X)$. Finally,  by Lemma \ref{lem: no isomorphism}, $\Sigma$ must be of finite  topological type. 
   \end{proof}
   
   As an immediate corollary we have 
   \begin{cor}\label{cor: MCG does not act}
   Suppose $\Sigma$ is a  general topological surface of negative Euler characteristic.  If $\Map(\Sigma)$ is isomorphic to $\Mod(X)$ for some Riemann surface
   $X$ then $\Sigma$ has finite topological type and X is homeomorphic to $\Sigma$.  Moreover, a big mapping class group can not act on any Teichm\"uller
   space with orbits equivalent to modular group orbits. 
   \end{cor}
   
   \begin{proof}
   	Let $\rho$ be an isomorphism from $\Map(\Sigma)$ to $\Mod(X)$ for some Riemann surface $X$. By lemma \ref{lem: no isomorphism} $\Sigma$ can not have infinite topological type. On the other hand,  if $\Sigma$ has finite topological type then $\Map(\Sigma)$ is equal to (and hence)  isomorphic to  $\Mod(Y)$ for a Riemann surface  $Y$  homeomorphic to $ \Sigma$.  But $\rho$ being an isomorphism implies $\Mod(X)$ is isomorphic to $\Mod(Y)$ and the algebraic rigidity of finitely generated modular groups \cite{Iv} implies $X$ is homeomorphic to $Y$.

   	The failiure to act on any Teichm\"uller space follows immediately from \ref{thm: MCG of finite type}.

   \end{proof}

  The question of whether a big mapping class group can act on
  a Teichm\"uller space appears      in \cite{OpQs}.

%%%%%%%%%%%%%%%%%%%%%%%%%%%%

\section{Relative trichotomy} \label{sec: rel. trichotomies}

In this section we  construct a number of subspaces in $H(\Sigma)$ and illustrate how 
 mapping classes in the trichotomy can change type depending on the subspace. Throughout we make essential use of the lower and upper  bounds of   Theorem 1 in    \cite{Ma1}.  

\subsection{Systoles and multi-twists:}  We define the {\it systole} of a hyperbolic surface to be 
$$
\inf \{\ell_{\gamma}(X): \gamma  \text{ is a closed geodesic} \}.
$$ 
This infimum may or may not exist depending on the surface.  
\label{sec: relative tri}
\begin{prop} \label{prop: multi-twists}

Let $P$ be the property that the hyperbolic surface  $X$ has systole bounded from below by a positive constant.  Consider the subspace 
$\mathcal{C}_{P}  \subset H(\Sigma)$. 
\begin{itemize}
\item For any unbounded sequence    $\{n_i \in \ZZ\}$, the infinite multi-twist 
$ \prod_{i \in I} D_i ^{n_i}$  is never qc relative to $\mathcal{C}_{P}$.
\item  For any bounded sequence    $\{n_i \in \ZZ\}$ with infinitely many $n_i \neq 0$,  the infinite multi-twist 
$ \prod_{i \in I} D_i ^{n_i}$  is sometimes  qc relative to $\mathcal{C}_{P}$.
\end{itemize}
\end{prop}

\begin{proof}
We use  Theorem 1 inequalities of \cite{Ma1} throughout this proof. 
For an unbounded sequence, since the length of any closed geodesic is bounded from below, the Matsuzaki  lower bound 
goes to infinity.  

For a bounded sequence,  there exists 
$X \in \mathcal{C}_{P}$ with  the lengths of the $\gamma_{i}$ equal to  $1$. In this case, the  Matsuzaki  upper bound  is finite.  On the other hand, if 
the lengths of the $\gamma_{i}$ go to infinity then there is no qc representative 
for the multi-twist. 
\end{proof}

\begin{lem} \label{lem: uncountably many}
Let $(X,f)$ be a hyperbolic structure. If there is an infinite multi-twist
$\phi$ that is qc with respect to $(X,f)$, then there are uncountably many 
multi-twists which are qc with respect to $(X,f)$.
\end{lem}

Next, let $Q$ be the property that the Teichm\"uller modular group is countable.
Such hyperbolic surfaces exist (see \cite{Ma2}, \cite{Her}) for  any  topological type 
$\Sigma$.  Consider the subspace $\mathcal{C}_{Q} \subset H(\Sigma)$, and note that  $\mathcal{C}_{Q} \subset \mathcal{C}_{P}$ by Lemma  
\ref{lem: uncountably many}.

Any infinite  multi-twist
$ \prod_{i \in I} D_i ^{n_i}$   is  never qc relative to $\mathcal{C}_{Q}$. 
In particular, all  of the bounded sequences  from Proposition 
\ref{prop: multi-twists}  that were sometimes qc 
relative to  $\mathcal{C}_{P}$ have become never qc with respect to 
$\mathcal{C}_{Q}$. 
   
   These multi-twist examples underscore the principle that if  we have nested subspaces $\mathcal{C}_{1}  \subset \mathcal{C}_{2}  \subset H(\Sigma)$
   then 
 $$\{\phi  \in \Map (\Sigma): \text{sometimes qc rel. to } \mathcal{C}_{1}\}
 \subset \{\phi  \in \Map (\Sigma): \text{sometimes qc rel. to } \mathcal{C}_{2}\}
 $$
and 
$$\{\phi  \in \Map (\Sigma): \text{never  qc rel. to } \mathcal{C}_{2}\}
 \subset \{\phi  \in \Map (\Sigma): \text{never  qc rel. to } \mathcal{C}_{1}\}
 $$

 \subsection{Nested $\Map$-invariant subspaces:} 
 In this section, we construct a nested decreasing sequence of subspaces 
 in $H(\Sigma)$ so that  for each adjacent pair subspaces 
  there is a mapping class  which is sometimes qc relative to the bigger subspace but  never qc with respect to the smaller one. 
 
 Fix $\Sigma$ an infinite type topological surface. We say that 
 a hyperbolic surface has a {\it   bounded pants decomposition}  if it has a pants decomposition for which the pants curves have length bounded above and below    by  positive constants. If there is only a bound from above we say the pants decomposition is  {\it upper bounded}.   Recall that $\mathcal{C}$ denotes the geodesically complete structures on $\Sigma$.

 For $r \in \mathbb{N} \cup \{\frac{1}{n}\}_{n=1}^{\infty}$, consider the  subspace   
 $\mathcal{D}_{r} \subset   \mathcal{C}$ of hyperbolic structures  
 $(X,f)$ where   $\Sigma$ has  a  pants decomposition $\{\gamma_{m}\}_{m=1}^{\infty}$ and there are constants $c,C>0$ so that   
$$
c \frac{1}{m^{1/r}} \leq \ell_{X}(f(\gamma_m)) \leq C
$$
for  all $m$. 
 
 Note that such a pants curve condition is a quasiconformal invariant  by Wolpert's lemma (see the references in \cite{Bus}) and these subspaces form a nested decreasing sequence
 $$
 \mathcal{C} \supset  \mathcal{D}_{0} \supset  \cdots  \supset
 \mathcal{D}_{\frac{1}{n}} \supset \cdots \supset   
 \mathcal{D}_{\frac{1}{2}}   \supset 
   \mathcal{D}_{1}    \supset  \cdots  \supset \mathcal{D}_{n} \supset \cdots
 $$

 \begin{thm}  For $r \in \mathbb{N} \cup \{\frac{1}{n}\}_{n=1}^{\infty}$,  the subspace $\mathcal{D}_{r}$ is  a locally path connected, connected  $\Map$-invariant subspace  where  
 \begin{enumerate}
 \item $\mathcal{D}_{\infty}:=\cap_{n=1}^{\infty} \mathcal{D}_{n}$ 
  is  locally path connected, connected and 
  $$\mathcal{D}_{\infty}=\{(X,f) \in \mathcal{C}: X 
 \text{ has a  bounded pants decomposition} \}
 $$

  \item $\mathcal{D}_{0}:=\cup_{n=1}^{\infty} \mathcal{D}_{\frac{1}{n}}$
  is  locally path connected, connected  and 
 $$\mathcal{D}_{0}=\{(X,f) \in \mathcal{C}: X 
 \text{ has an upper   bounded pants decomposition} \}
 $$
 \item For each $n=1,2,3,...$, there exists an element  $\phi$ in 
 $\Map$ so that $\phi$ relative to  $\mathcal{D}_{n}$ is sometimes qc but relative to $\mathcal{D}_{n+1}$ is never qc. 
 Similarly,   there exists an element  $\phi$ in 
 $\Map$ so that $\phi$ relative to  $\mathcal{D}_{\frac{1}{n+1}}$ is sometimes qc but relative to $\mathcal{D}_{\frac{1}{n}}$ is never qc. 

 \end{enumerate}
 \end{thm}
 
 \begin{proof} The local path connected, connectedness of each of the subspaces follows from Theorem \ref{thm: property P connectivity} since the 
 pants decomposition condition  for each subspace is a quasi-invariant quantity, and the characterization of $\mathcal{D}_{0}$ and 
 $\mathcal{D}_{\infty}$ in items (1) and (2) follow from the construction. 
 
 To prove item (3) fix a positive integer $n$ and a pants decomposition 
 $\{\gamma_{m}\}$ of $\Sigma$.  Define $\phi \in \Map (\Sigma)$ to be  the multi-twist about the pants curves  where the twist about $\gamma_m$ is
 $n_{\gamma_m}:=m^{\frac{1}{n}}$. We consider the action of $\phi$ on 
 $\mathcal{D}_{n}$. Let the hyperbolic structure $(X,f) \in  \mathcal{D}_{n}$,   be such  that 
 $$
 \ell_{X}(f(\gamma_{m})) = \frac{1}{m^{\frac{1}{n}}}, \text{ for each } m.
 $$
 Then since  $n_{\gamma_m} \ell_{X}(f(\gamma_{m}))=1$ for each $m$,
  by Matsuzaki's upper bound condition (Theorem 1 in    \cite{Ma1})
  we see that $\phi$ is qc with respect to the hyperbolic structure $(X,f)$.  On the other hand, for the hyperbolic structure $(Y,g) \in  \mathcal{D}_{n}$, where $\ell_{Y}(g(\gamma_{m}))=C$ $\phi$ is qc since 
  $$
  n_{\gamma_m} \ell_{Y}(g(\gamma_{m}))=m^{\frac{1}{n}} C
  \rightarrow \infty, \text{ as } m \rightarrow \infty
  $$
  by Matsuzaki's lower bound inequality (Theorem 1 in    \cite{Ma1})
  we have that $\phi$ is not qc with respect to $(Y,g)$. Hence we have shown $\phi$ relative to $\mathcal{D}_{n}$
  is sometimes qc.
  
  Next consider the action of $\phi$ on $\mathcal{D}_{n+1}$.  If 
  $(X,f) \in \mathcal{D}_{n+1}$, then by definition there exists a pants decomposition $\{\beta_{m}\}$ of $\Sigma$ so that 
  $$
   c \frac{1}{m^{\frac{1}{n+1}}}  \leq \ell_{X} (f(\beta_{m})) \leq C
  $$
  
 There are two cases to consider:
 \begin{enumerate}
 \item  The multi-twist curves $\{\gamma_m \}$ are eventually the same as 
  the pants curves $\{\beta_{m}\}$, then we have 
  $$
  n_{\gamma_{m}} \ell_{X} (f(\gamma_{m})) 
  \geq          \frac{c m^{\frac{1}{n}}}{m^{\frac{1}{n+1}}} \rightarrow \infty,
  \text{ as } m \rightarrow \infty.
  $$
  Hence by Matsuzaki (Theorem 1 in    \cite{Ma1})  $\phi$  is not qc with respect to $(X,f)$. The other possibility is that the pants decomposition $\{\beta_{m}\}$ 
 \item  The multi-twist curves $\{\gamma_m \}$ are not  the same as 
  the pants curves $\{\beta_{m}\}$ for $m$ large enough. But  then there infinitely many  pants curves $\{\beta_{m} \}$  transverse to the $\{\gamma_{m}\}$. Then since $\ell_{X} (f(\beta_{m})) \leq C$, we have 
  $\ell_{X} (f(\gamma_{m})) \geq K$, for some $K$.  Then as in the previous case, 
  $$
    m^{\frac{1}{m}} K  \leq n_{\gamma_{m}} \ell_{X} (f(\gamma_{m})) 
    \rightarrow  \infty, \text{ as } m \rightarrow \infty. 
  $$
 \end{enumerate}
 
 Thus $\phi$ relative to $\mathcal{D}_{n+1}$ is never qc. 
 
 The argument for the reciprocal integers is the same. We leave the details to the reader. 
 \end{proof}
 
 \begin{rem}
 The lower bound functions in the above theorem have been chosen for convenience as there are many other functions that could have been used 
 to achieve  the  construction. 
 \end{rem}

%
%\section{Sufficient topological conditions for never quasiconformal}
%In this section, we give several sufficient topological conditions for a mapping class to be never quasiconformal. 
%
%\subsection{Twisting and intersection condition}
%
%\begin{lem}
%	If there exists a sequence of disjoint simple closed curves $\{\gamma_i\}$ such that $i(\gamma_i \phi(\gamma_i)) \cdot twist_{\gamma_i}(\phi(\gamma_i))$ is unbounded then $\phi$ is never quasiconformal.
%\end{lem} 

\bibliographystyle{plain} % We choose the "plain" reference style
\bibliography{references} % Entries are in the references.bib file

\end{document}